\documentclass[12pt]{amsart}
\usepackage[all,cmtip]{xy}
\usepackage{amsaddr}
\usepackage{amsmath, amsthm, amssymb}
\usepackage{eucal, mathrsfs}
\usepackage{amsfonts}
\usepackage{array}
\usepackage{hyperref}
\usepackage{graphicx}
\usepackage{caption}
\usepackage{subcaption}
\usepackage{enumerate}
\usepackage{verbatim}
\usepackage{color}

\newtheorem*{thm*}{Theorem}
\newtheorem{thm}{Theorem}[section]

\newtheorem{lem}[thm]{Lemma}
\newtheorem{prop}[thm]{Proposition}

\theoremstyle{definition}
\newtheorem{defn}[thm]{Definition}
\theoremstyle{remark}
\newtheorem{remark}[thm]{Remark}
\numberwithin{equation}{section}

\newcommand{\p}{\ensuremath{\partial}}
\newcommand{\dd}[2][]{\ensuremath{\dfrac{d^{#1}}{d{#2}^{#1}}}}

\newcommand{\abs}[2][]{\ensuremath{\left|{#2}\right|_{#1}}}
\newcommand{\norm}[2][]{\ensuremath{\left\|{#2}\right\|_{#1}}}

\begin{document}

\title[The Multiplier Problem for Scalar ODE's]{The Multiplier Problem of the Calculus of Variations for Scalar Ordinary Differential Equations}

\author{Hardy Chan}
\thanks{This work is partially supported by the Four Year Fellowship of The University of British Columbia. The author is deeply indebted to his M.Phil. advisor Prof. Kai-Seng Chou for suggesting this problem, as well as carefully reviewing the first draft of this article and giving numerous valuable suggestions. The author also thanks Prof. Kai-Seng Chou for his continuous support, while and after pursuing his master degree at The Chinese University of Hong Kong.} 
\address{Department of Mathematics, The University of British Columbia}



\begin{abstract}
In the inverse problem of the calculus of variations one is asked to find a Lagrangian and a multiplier so that a given differential equation, after multiplying with the multiplier, becomes the Euler--Lagrange equation for the Lagrangian.  An answer to this problem for the case of a scalar ordinary differential equation of order $2n, n\geq 2,$ is proposed.
\end{abstract}

\maketitle

\vspace{1cm}

\noindent {\bf Keywords} The inverse problem to the calculus of variations, variational multiplier, Euler--Lagrange operator, Fels' conditions\\
\noindent {\it 2000 Mathematics Subject Classification} Primary
53B50, 49N45, Secondary 35K25.

\tableofcontents

\section{Introduction}
\subsection{The Multiplier Problem}
Let $L(x,p_0,p_1,\dots, p_n)$ be a given smooth function. The Euler--Lagrange equation for the Lagrangian $L(x,u, u',\dots,u^{(n)})$ is given by
\begin{equation}\label{1.1}
EL[u](x):=\sum_{k=0}^n (-1)^k\dfrac{d^k }{dx^k}L_{p_k}(x,u(x),\dots,u^{(n)}(x))=0.
\end{equation}
This is a differential equation with leading term $L_{p_np_n}(x,u,\dots,u^{(n)})u^{(2n)}$.  The multiplier problem in the calculus of variations asks a converse question.  Given a differential equation, say,
\begin{equation}\label{1.2}
u^{(2n)}(x)-f\left(x,u(x),\dots, u^{(2n-1)}(x)\right)=0.
\end{equation}
When is there a Lagrangian $L(x,p_0,p_1,\dots,p_n)$ of order $n$ and some positive multiplier $\rho(x)$ such that
\begin{equation}\label{1.3}
EL[u](x)= \rho(x)\big(u^{(2n)}-f(x,u(x),\dots,u^{(2n-1)}(x))\big)
\end{equation}
holds for all smooth functions $u$?

By comparing this equation with \eqref{1.1}, we know that the variational multiplier $\rho$ must have the same order of dependence as the Lagrangian $L$ when $L$ is non-degenerate, that is, $L_{p_np_n}(x,p_0,\dots,p_n)$ is non-zero.  Using the fact that \eqref{1.3} must hold for all functions, \eqref{1.3} is essentially a single partial differential equation for two unknowns $L$ and $\rho$ in $(x,p_0,\dots, p_n)$. (A precise formulation would require extra independent variables $p_{n+1},\dots,p_{2n}$ to take care of the higher derivatives of $u$.)

Later we will see the multiplier is determined by the Lagrangian and the Lagrangian satisfies an overdetermined system of linear partial differential equations when $n\geq 2$.  The solvability of this overdetermined system depends on the order, or now the number of independent variables in it.  When the order is two, that is, $n=1$, the equation for $L$ is not overdetermined and a variational multiplier always exists. This is a classical result of Darboux\cite{D}.  The next case $n=2$ comes much later.  Adapting the modern approach by Anderson--Thompson\cite{AT}, in 1996 Fels\cite{F} discovers two quantities involving the partial derivatives of $f$ and shows that their vanishing is necessary and sufficient for the existence of a variational multiplier for a fourth order equation.  Subsequently, the problem is solved for sixth and eighth order equations by Jur\'{a}\v{s}\cite{J1} in the spirit of \cite{F}.  Instead of two vanishing quantities, there are three and five quantities for sixth and eighth order equations respectively.  It is expected there are more and more vanishing quantities in the necessary and sufficient conditions as the order of the equation increases.

In this paper we propose a solution to the multiplier problem for all $2n$-order equations.  We will show that \eqref{1.3} possesses a solution if and only if $f$ assumes a certain integral form involving with some Euler--Lagrange operators (defined below) and $n+2$ many free functions as parameters.  Moreover, when this happens, $L$ and $\rho$ are also given in integral forms depending on these operators and  free functions.  This is different from all previous works.  The reason lies on the methodology; in contrast to sophisticated tools such as Cartan's equivalence method and variational bicomplex employed in  \cite{F} and \cite{J1}, our approach is completely elementary.

\subsection{Notations}\label{notations}
To state our results, we first fix some notations. Throughout the paper, unless otherwise specified, the subscript of a function denotes the highest order of derivatives that it depends on. For example, a function $L_2$ is understood to be $L_2(x,p_0,p_1,p_2)$. Occasionally, a function depending only on $x$ is given a subscript $-1$. We write $\p_k=\p_{p_k}$.

For a smooth function $F$ depending on $p_k$ ($k\geq 0$), we denote the first and second anti-derivatives of $F$ with respect to $p_k$ respectively by
$$\int^{p_k}F:=\int_0^{p_k}F(x,p_1,\dots,p_{k-1},\tilde{p}_{k},p_{k+1},\dots)\,d\tilde{p}_k$$
and
$$\int^{p_k}\!\!\!\int\!F:=\int_0^{p_k}\!\!\!\int_0^{\tilde{p}_k}\!F(x,p_1,\dots,p_{k-1},\hat{p}_k,p_{k+1},\dots)\,d\hat{p}_kd\tilde{p}_k.$$
Similar notations apply to anti-derivatives with respect to $x$.

\begin{defn}[The differential operators]\label{Def1.1}
The \emph{truncated total differential operators} are
\begin{equation*}\begin{split}
D_0&:=\p_x, \\
D_m&:=\p_x+p_1\p_0+p_2\p_1+\cdots+p_m\p_{m-1} \quad (m\geq1),
\end{split}\end{equation*}
and the \emph{$m$-th order Euler--Lagrange operator with $(n+1)$ terms} is defined by
\begin{equation*}
E_m^n := \sum_{k=0}^n(-1)^k D_m^k\p_k \quad (m,n\geq0).
\end{equation*}
(Note that the superscript in $D_m^k$ is a power while that in $E_m^n$ is not.)
\end{defn}

It is clear from the definition that $\p_n$ and $D_m$ do not commute if $m\geq n>0$. In fact, by the product rule,
$$\p_n D_m=D_m\p_n+\p_{n-1}.$$
This formula will be used freely throughout this paper.

Note that the Euler--Lagrange equation for the Lagrangian $L_n$ of order $n$ is given by $E^{n}_{2n}L_n=0$, see Lemma \ref{2.1} below.


In terms of these notations and according to Lemma \ref{2.1}, the multiplier problem is now expressed as, to find $L_n$ and $\rho_n$ such that
\begin{equation}\label{1.4}
\rho_n(p_{2n}-f_{2n-1})=E_{2n}^nL_n,
\end{equation}
holds for all $(x,p_0,p_1,\cdots,p_{2n})$.

\subsection{Motivations}
To gain some insight into this problem, let us consider the case $n=2$.  By expanding \eqref{1.4}($n=2$), we have
$$  \rho_2\big(p_4 - f_3 \big)=\Phi[L]+ p_4L_{p_2p_2},$$
where $\Phi[L]$ is a third order linear partial differential operator for $L$ given by
\begin{equation*}\begin{split}
\Phi[L]
&= L_{p_0} -\big(L_{xp_1} + p_1L_{p_0p_1} + p_2L_{p_1p_1} + p_3L_{p_1p_2}\big) \\
&\quad\, + \big(L_{xxp_2} + 2p_1L_{xp_0p_2} + 2p_2L_{xp_1p_2} + 2p_3L_{xp_2p_2} +p_1^2L_{p_0p_0p_2} + 2p_1p_2L_{p_0p_1p_2} \\
&\quad + 2p_1p_3L_{p_0p_2p_2} + p_2L_{p_0p_2} + p_2^2L_{p_1p_1p_2} + 2p_2p_3L_{p_1p_2p_2} + p_3L_{p_1p_2} + p_3^2L_{p_2p_2p_2} \big).
\end{split}\end{equation*}
By comparing the coefficient of $p_4$, we find that $\rho_2=L_{p_2p_2}$ and the remaining terms become a single linear partial differential equation for $L$:
$$\Phi[L] + f_3L_{p_2p_2}=0.$$
Observing that the independent variables in this equation are $(x,p_0,p_1,p_2,p_3)$ and yet we are looking for $L$ which depends only on $(x,p_0,p_1,p_2)$, there is another equation that $L$ must satisfy, namely, $L_{p_3}=0$.  Thus $L$ satisfies an overdetermined system of two linear partial differential equations.  In general, the Lagrangian $L_n, n\geq 3,$ satisfies $n$ many equations where one is the analog of $\Phi[L]=0$ and the others are $L_{p_{n+1}}=L_{p_{n+2}}=\cdots = L_{p_{2n-1}}=0.$ One needs to find out the constraints on $f_{2n-1}$ to solve \eqref{1.4}.

\medskip

\subsection{Main Results}

We first state our result for fourth order equations.

%
%
\begin{thm}\label{thmA}
Consider $n=2$ in \eqref{1.4}.  Suppose that \eqref{1.4}($n=2$) is solved for some $L_2$ and $\rho_2$.  There are functions $R_2, f_0, f_1$ and $N_1$ such that
\begin{equation}\label{1.5}
\left\{
\begin{array}{rcl}
f_3 & = & \displaystyle \p_2R_2\cdot p^2_3 + 2D_2R_2\cdot p_3 \\
& & \displaystyle - e^{R_2} \left(E^2_2 \int^{p_2} \!\!\! \int \! e^{-R_2} + f_1 p_2 -E^1_1 \int^{p_1} \!\!\! \int \! f_1+ f_0 \right)\ , \\
\rho_2 & = & e^{-R_2}\ , \\
L_2 & = & \displaystyle \int^{p_2} \!\!\! \int \! e^{-R_2} - \int^{p_1} \!\!\! \int \! f_1 + \int^{p_0} \! f_0 + D_2N_1 \ .
\end{array}
\right.
\end{equation}
Conversely, given any functions $R_2, f_0, f_1$ and $N_1$,   the functions $f_3, L_2,$ and $\rho_2$ defined in \eqref{1.5} solves \eqref{1.4}.
\end{thm}

For the general case, we have

%
%
\begin{thm}\label{thmB}
Consider \eqref{1.4} for $n\geq 3$.  Suppose that \eqref{1.4} is solved for some $L_n$ and $\rho_n$.  There are functions $R_n$, $(f_k)_{0\leq{k}\leq{n-1}}$ and $N_{n-1}$ such that
\begin{equation}\label{1.6}
\left\{
\begin{array}{rcl}
f_{2n-1} & = & \displaystyle nD_{n+1}R_n\cdot p_{2n-1} - e^{R_n}\left[(-1)^nE^n_{2n-2} \int^{p_n} \!\!\! \int \! e^{-R_n} \right.\\
& & \displaystyle \left. + \sum_{\ell=1}^{n-1} \left(f_{\ell} p_{2\ell} + (-1)^\ell E^{\ell}_{2\ell-1} \int^{p_\ell} \!\!\! \int \! f_\ell \right) + f_0 \right]\ , \\
\rho_n & = & e^{-R_n}\ , \\
L_n & = & \displaystyle (-1)^n \int^{p_n} \!\!\! \int \! e^{-R_n} + \sum_{\ell=1}^{n-1} (-1)^{\ell} \int^{p_\ell} \!\!\! \int \! f_{\ell} + \int^{p_0} \! f_0 + D_nN_{n-1}\ .
\end{array}
\right.
\end{equation}
Conversely, given any functions $R_n$, $(f_k)_{0\leq{k}\leq{n-1}}$ and $N_{n-1}$, the functions $f_{2n-1}, L_n$ and $\rho_n$ defined in \eqref{1.6} solves \eqref{1.4}.
\end{thm}

Theorem \ref{thmA} and Theorem \ref{thmB} are very similar except that $f_3$ is quadratic in $p_3$ while $f_{2n-1}, n\geq 3,$ is linear in $p_{2n-1}$.

This paper is organized as follows.  In Section 2 we collect some elementary while useful properties of the differential operators $D_m$ and $E_m^n$. In Section 3 we give a detailed proof of Theorem \ref{thmA}.  Since the form of the solution is quite different from those obtained in \cite{F}, we verify that the solution given here satisfies Fels' conditions in Section 4.  We treat the general case in Section 5.  In Section 6 we establish the analogues of Theorems \ref{thmA} and \ref{thmB} when the Lagrangian could be of order higher
than $n$.  Based on the form of our solution, we present in Section 7 an algorithm to check whether a given equation \eqref{1.2} has a variational multiplier or not. In the last section we mention an application in the study of parabolic partial differential equations.

We refer the reader to Anderson--Thompson \cite{AT} for a clear account of the background and results on the multiplier problem for ordinary differential equations.  Related results can be found in \cite{DZ}, \cite{J2} and \cite{NA}.  In general, the multiplier problem belongs to the inverse problem of the calculus of variations and it makes sense for partial differential equations.  The reader may consult Olver\cite{O} and Saunders\cite{S} for further information.  The works Anderson--Duchamp\cite{AD1} and \cite{AD2} contain results on the
multiplier problem for partial differential equations.

\medskip

%
%
\section{Auxiliary Results}
\subsection{Properties of the Differential Operators}

We observe that the variational multiplier problem is formulated in terms of the operator $E_{2n}^n$.

\begin{lem}\label{2.1}
Given a smooth function $L=L_n(x,p_0,\dots,p_n)$, we have
$$EL[u]=(E_{2n}^nL_n)[u].$$
\end{lem}

\begin{proof}
By linearity, it suffices to show that for $k=0,1,\dots,n$,
\begin{equation*}
\dd[k]{x}L_n[u]=(D_{2n}^kL_n)[u].
\end{equation*}
Since the left hand side depends on derivatives up to order $n+k\leq2n$, each $d/dx$ can be replaced by $D_{2n}$.
\end{proof}

In the following we list some elementary properties and auxiliary results for the fourth order problem, i.e. when $n=2$.

\begin{lem}\label{2.2}
The followings hold true.
\begin{enumerate}[(a)]
\item (Extracting the main order) For $m\geq1$,
$$D_m=D_{m-1}+p_m\p_{m-1}.$$
\item (Commutator relation) For $m\geq{n}\geq1$,
$$\p_{n}D_m=D_{m}\p_n+\p_{n-1}.$$
\item (Reducing the Euler--Lagrange operators)
\begin{equation*}\begin{split}
E_4^2&=E_3^2+(-p_4\p_1+2p_4D_3+p_4^2\p_3)\p_3+p_4\p_2^2\\
E_3^2&=E_2^2+2p_3D_2\p_2^2+p_3^2\p_2^3.
\end{split}\end{equation*}
\item (A telescoping sum)
$$E_2^1D_2=-D_2^2\p_1.$$
\end{enumerate}
\end{lem}

\begin{proof}
Parts (a) and (b) follow from  Definition \ref{Def1.1}. (c) and (d) can be obtained by direct computations.
\end{proof}

\subsection{Higher Order Analogues}

In this subsection we deal with those results needed for the general case $n\geq3$.

When we replace $D_m$ by its power in Lemma \ref{2.2}(b), we obtain a binomial-type formula.
\begin{lem}[Commutating with powers]\label{2.3}
For $k,m,n\geq0$ and $m\geq{n}$, we have
$$\p_nD_m^k=\sum_{j=0}^{\min\{n,k\}}\binom{k}{j}D_m^{k-j}\p_{n-j}.$$
\end{lem}

The proof, which resembles the standard inductive proof of the binomial theorem, is postponed to the appendix.

The generalization is very similar if one has higher partial derivatives on the left. For our purpose, one with a second derivative is enough.

\begin{lem}\label{2.3b}
For $m\geq{n}\geq1$,
$$\p_n^2D_m=(D_m\p_n+2\p_{n-1})\p_n.$$
\end{lem}
The proof is straightforward.

To state a formula for the pure power $D_m^k$, we employ a multi-index notation. Let $m\geq1$ be an integer. For an $m$-tuple of non-negative integers $i_{m-1},i_{m-2},\dots,i_0$, we write
$$I:=(i_{m-1},i_{m-2},\dots,i_0).$$
The descending labeling is due to the definition $$\p^{I}:=(\p_{m-1},\p_{m-2},\dots,\p_0)^{I}=\p_{m-1}^{i_{m-1}}\p_{m-2}^{i_{m-2}}\cdots\p_0^{i_0}.$$
We say that $I\geq0$ if $i_{m-j}\geq0$ for all $j=1,\dots,m$. The factorial of $I$ is defined to be
$$I!:=\prod_{j=1}^{m}i_{m-j}!=i_{m-1}!\cdots i_0!$$
and its absolute value is defined to be
$$\abs{I}:=\sum_{j=1}^{m}i_{m-j}=i_{m-1}+i_{m-2}+\cdots+i_0.$$
We also define its weighted additive norm by
$$\norm{I}:=\sum_{j=1}^{m}ji_{m-j}=i_{m-1}+2i_{m-2}+\cdots+mi_0$$
and its weighted multiplicative norm by
$$\norm{I}^*:=(1!,\dots,m!)^I=\prod_{j=1}^{m}(j!)^{i_{m-j}}=(1!)^{i_{m-1}}(2!)^{i_{m-2}}\cdots(m!)^{i_0}.$$
From now on, in any summation over a multi-index $I$, we implicitly require $I\geq0$.

\begin{lem}[Main terms of a power]\label{2.4}
For $m>k\geq1$,
$$D_m^k=\sum_{\norm{I}\leq{k}}a_I^{(k)}p_m^{\abs{I}}D_{m-1}^{k-\norm{I}}\p^I,$$
where
\begin{equation*}
a_I^{(k)}=
\begin{cases}
\dfrac{k!}{\norm{I}^{*}I!(k-\norm{I})!},&\text{if}\,I\geq0\\
0,&\text{otherwise}.
\end{cases}
\end{equation*}
\end{lem}

A simple while non-trivial example is
$$D_4^2=D_3^2+2p_4D_3\p_3+p_4^2\p_3^2+p_4\p_2,$$
in the case $n=2$, where $I$ ranges over
$$(0,0,0,0),\,(1,0,0,0),\,(2,0,0,0),\text{ and }(0,1,0,0).$$
For general $n$, in analogy, the main terms are those where $\norm{I}$ is small. By the definition of $\norm{I}$, one may write down explicitly the terms involved and use it to solve the multiplier problem.

A detailed analytic proof as well as a combinatorial proof are given in the appendix.

The analogue of Lemma \ref{2.2}(d) is similar.
\begin{lem}[A telescoping sum]\label{2.5}
For $m\geq{n}\geq1$,
$$E_{m}^{n}D_m=(-1)^{n}D_m^{n+1}\p_n.$$
\end{lem}

\begin{proof}
We compute, by definition and Lemma \ref{2.2}(a),
\begin{align*}
E^n_m D_m =\, & \sum_{k=0}^n (-1)^k D^k_m \p_k D_m \\
=\, & \p_0 D_m + \sum_{k=1}^n (-1)^k D^k_m (D_m \p_k + \p_{k-1}) \\
=\, & D_m \p_0 + \sum_{k=1}^{n-1} (-1)^k D^{k+1}_m \p_k + (-1)^n D^{n+1}_m \p_n + \sum^n_{k=1} (-1)^k D^k_m \p_{k-1} \\
=\, & \sum^{n-1}_{k=0} (-1)^k D^{k+1}_m \p_k + (-1)^k D^{n+1}_m \p_n + \sum^{n-1}_{k=0} (-1)^{k+1} D^{k+1}_m \p_k \\
=\, & (-1)^n D^{n+1}_m \p_n.
\end{align*}
\end{proof}

\subsection{Solvability of the Linear Equation}

The following lemma is a key step to solving the multiplier problem once the main term is taken out. We recall that the subscript notation is given in Section \ref{notations}.

\begin{lem}[Integrating the linear equation]\label{2.6}
Suppose $n\geq-1$. The equation
$$D_{n+1}\rho_n+g_{n+1}\rho_{n}=0$$
has a solution with $\rho_n>0$ if and only if there exists a function $R_n$ such that
\begin{equation*}
\left\{
\begin{array}{rcl}
 \rho_n & = & e^{-R_n} \\
g_{n+1} & = & D_{n+1} R_n\ .
\end{array}
\right.
\end{equation*}
\end{lem}

\begin{proof}
The ``if'' part follows from the exponential chain rule,
$$D_{n+1}e^{-R_n}=e^{-R_n}D_{n+1}(-R_n),$$
which is valid because of the linearity of $D_{n+1}$ and the chain rule for partial derivatives.

We prove the ``only if'' part by induction on $n$. For $n=-1$, suppose $D_0 \rho_{-1} + g_0 \rho_{-1} = 0$ has a solution $\rho_{-1}(x)>0$, then $g_0=-\p_x\rho_{-1}/\rho_{-1}$ depends only on $x$. In this case, $\displaystyle\rho_{-1}= C \exp\left(-\int^x\!g_0\right)$ for some positive constant $C$. We can simply take $\displaystyle R_{-1}=-\int^x\!g_0-\log C$.

Assume the result holds for $n$ and let $D_{n+2} \rho_{n+1} + g_{n+2} \rho_{n+1} = 0$ have a solution $\rho_{n+1}>0$. Since the first term is linear in $p_{n+2}$, applying $\p_{n+2}$ twice to the equation yields
\begin{equation*}
\p^2_{n+2} g_{n+2} = 0.
\end{equation*}
Write $g_{n+2} = \alpha_{n+1} p_{n+2} + \beta_{n+1}$. Comparing the coefficients of $p_{n+2}$ in $$(D_{n+1}+p_{n+2}\p_{n+1})\rho_{n+1}+(\alpha_{n+1}p_{n+2}+\beta_{n+1})\rho_{n+1}=0,$$ we obtain
\begin{equation*}\begin{cases}
\p_{n+1} \rho_{n+1} + \alpha_{n+1} \rho_{n+1} = 0 \\
  D_{n+1} \rho_{n+1} + \beta_{n+1} \rho_{n+1} = 0.
\end{cases}\end{equation*}
Solving the first equation, $\displaystyle \rho_{n+1} = \exp \left( -\int^{p_{n+1}} \! \alpha_{n+1} \right) \rho_n $ for some positive function $\rho_n>0$. Substituting this into the second one, we have
\begin{align*}
\exp \left( -\int^{p_{n+1}} \! \alpha_{n+1} \right) D_{n+1} \rho_n - & \exp \left( - \int^{p_{n+1}} \! \alpha_{n+1} \right) D_{n+1} \int^{p_{n+1}} \! \alpha_{n+1} \cdot \rho_n \\
& + \beta_{n+1} \exp \left( - \int^{p_{n+1}} \! \alpha_{n+1} \right) \rho_n = 0,
\end{align*}
or
\begin{equation*}
D_{n+1} \rho_n + \left( \beta_{n+1} - D_{n+1} \int^{p_{n+1}} \! \alpha_{n+1} \right) \rho_n = 0.
\end{equation*}
By the induction hypothesis, there exists some $R_n$ such that $\rho_n = e^{-R_n} $ and
\begin{equation*}
D_{n+1} R_n = \beta_{n+1} - D_{n+1} \int^{p_{n+1}} \! \alpha_{n+1}.
\end{equation*}
Taking  $\displaystyle R_{n+1} = R_n + \int^{p_{n+1}} \! \alpha_{n+1}$, we have $\rho_{n+1} = e^{-R_{n+1}}$ and
\begin{equation*}
\begin{split}
D_{n+2} R_{n+1}
&= \alpha_{n+1} p_{n+2} + D_{n+1} \left( R_n + \int^{p_{n+1}} \! \alpha_{n+1} \right) \\
&= \alpha_{n+1} p_{n+2} + \beta_{n+1} \\
&= g_{n+2}.
\end{split}
\end{equation*}
By induction,  the result holds for all $n\geq-1$.
\end{proof}

%
%
\section{The Fourth Order Equation}

Ideas in the proof of Theorem \ref{thmA} can be outlined as follows. First, $\rho_2(x,p_0,p_1,p_2)$ can be expressed in terms of $L_2(x,p_0,p_1,p_2)$ by comparing the coefficients of $p_4$ on both sides of \eqref{1.4}($n=2$). Then, by expanding the right hand side of \eqref{1.4}, the highest order terms, that is, those terms involving $p_3$, in $f_3$ can be obtained. Eventually the form of $f_3$ can be revealed completely by further extracting the highest order terms in the Euler--Lagrange operators.

\begin{proof}[Proof of Theorem \ref{thmA}]
We will establish the equivalence of the following four statements.  Theorem \ref{thmA} follows by the equivalence of (a) and (d).
\begin{enumerate}[(a)]
\item There exist $\rho_2$, $f_3$ and $L_2$ such that \eqref{1.4}($n=2$) holds, that is,
    \begin{equation}\label{3.1}
    \rho_2(p_4-f_3)=E_4^2L_2.
    \end{equation}
\item There exist $R_2$, $L_1$ and $\widetilde{L}_1$ such that
    \begin{equation}\label{3.2}
    \left\{
    \begin{array}{rcl}
    \rho_2 & = & e^{-R_2}\ , \\
       f_3 & = & \p_2R_2\cdot p_3^2+2D_2R_2\cdot p_3 \\
           &   & \displaystyle -e^{R_2}\left[E_2^2\int^{p_2}\!\!\!\int\!e^{-R_2}+E_2^2\left(L_1p_2+\widetilde{L}_1\right)\right]\ , \\
       L_2 & = & \displaystyle \int^{p_2}\!\!\!\int\!e^{-R_2}+L_1p_2+\widetilde{L}_1 \ .
    \end{array}
    \right.
    \end{equation}
\item There exist $R_2$, $f_1$, $L_1$, $L_0$ and $\widetilde{L}_0$ such that
    \begin{equation}\label{3.3}
    \left\{
    \begin{array}{rcl}
    \rho_2 & = & e^{-R_2}\ , \\
       f_3 & = & \p_2R_2\cdot p_3^2+2D_2R_2\cdot p_3 \\
           &   & \displaystyle -e^{R_2}\left[E_2^2\int^{p_2}\!\!\!\int\!e^{-R_2}+f_1p_2-E_1^1\int^{p_1}\!\!\!\int\!f_1+E_1^1\left(L_0p_1+\widetilde{L}_0\right)\right]\ , \\
       L_2 & = & \displaystyle \int^{p_2}\!\!\!\int\!e^{-R_2}-\int^{p_1}\!\!\!\int\!f_1+D_2\int^{p_1}\!L_1+L_0p_1+\widetilde{L}_0 \ .
    \end{array}
    \right.
    \end{equation}
\item There exist $R_2$, $f_1$, $f_0$ and $N_1$ such that (1.5) holds, that is,
    \begin{equation*}
    \left\{
    \begin{array}{rcl}
    \rho_2 & = & e^{-R_2} \ ,\\
       f_3 & = & \p_2R_2\cdot p_3^2+2D_2R_2\cdot p_3 \\
           &   & \displaystyle -e^{R_2}\left[E_2^2\int^{p_2}\!\!\!\int\!e^{-R_2}+f_1p_2-E_1^1\int^{p_1}\!\!\!\int\!f_1+f_0\right] \ , \\
       L_2 & = & \displaystyle \int^{p_2}\!\!\!\int\!e^{-R_2}-\int^{p_1}\!\!\!\int\!f_1+\int^{p_0}\!f_0+D_2N_1 \ .
    \end{array}
    \right.
    \end{equation*}
\end{enumerate}
Before providing the details, we point out that the equivalence of (a) and (b) already yields a solution of the multiplier problem. However, \eqref{1.5} gives a more explicit form for $f_3$, using which we can check whether or not \eqref{1.4}($n=2$) has a solution by the algorithm we state in the last section.

(a) $\Rightarrow$ (b): Let $L_2$ and $\rho_2$ solve (2.1). To extract the terms involving $p_4$, we recall from Lemma \ref{2.2}(c) that
$$E_4^2=E_3^2+(-p_4\p_1+2p_4D_3+p_4^2\p_3)\p_3+p_4\p_2^2.$$
As $\p_3L_2=0$, we have $E_4^2L_2=E_3^2L_2+p_4\p_2^2L_2$ and so \eqref{3.1} is decoupled to the system
\begin{equation*}
\left\{
\begin{array}{rcl}
     \rho_2 & = & \p^2_2 L_2 \ ,\\
-\rho_2 f_3 & = & E_3^2 L_2\ .
\end{array}
\right.
\end{equation*}
To see the dependence of $f_3$ on $p_3$, we use the second equation in Lemma \ref{2.2}(c) to obtain
\begin{equation*}
-\rho_2f_3=E_2^2L_2+2p_3D_2\p_2^2L_2+p_3^2\p_2^3L_2.\\
\end{equation*}
Differentiating both sides with respect to $p_3$, we have
\begin{eqnarray*}
-\rho_2 \p_3 f_3 & = & 2D_2 \p^2_2 L_2 + 2p_3 \p^3_2 L_2\\
& = & 2D_3 \p^2_2 L_2\\
& = & 2D_3\rho_2
\end{eqnarray*}
by the above system, or
\begin{equation*}
D_3 \rho_2 + \left( \dfrac{1}{2} \p_3 f_3 \right) \rho_2 = 0.
\end{equation*}
Applying Lemma \ref{2.6}, we see that this equation is solvable if and only if there exists a function $R_2(x,p_0,p_1,p_2)$ such that $\rho_2 = e^{-R_2}$ and
\begin{equation}\label{thmApf:compare}
\dfrac{1}{2}\p_3f_3=D_3R_2=p_3\p_2R_2+D_2R_2.
\end{equation}
Hence,
\begin{equation*}
f_3 = \p_2 R_2\cdot p_3^2 + 2D_2 R_2\cdot p_3 - e^{R_2} E^2_2 L_2.
\end{equation*}
From $\p^2_2 L_2 = e^{-R_2}$, there exist functions $L_1(x,p_0,p_1)$ and $\widetilde{L}_1(x,p_0,p_1)$ such that
\begin{equation*}
L_2 = \int^{p_2} \!\!\! \int \! e^{-R_2} + L_1 p_2 + \widetilde{L}_1,
\end{equation*}
and so
\begin{equation*}
f_3 = \p_2 R_2 \cdot p^2_3 + 2D_2 R_2 \cdot p_3 - e^{R_2} \left[ E^2_2 \int^{p_2} \!\!\! \int \! e^{-R_2} + E^2_2 \left(L_1 p_2 + \widetilde{L}_1 \right) \right].
\end{equation*}

(b) $\Rightarrow$ (a): Putting \eqref{3.2} into \eqref{3.1} and using the above calculations, we have
\begin{equation*}
\begin{split}
 &E_4^2L_2\\
=&E_3^2L_2+p_4\p_2^2L_2\\
=&E_2^2L_2+2p_3D_2\p_2^2L_2+p_3^2\p_2^3L_2+p_4\p_2^2L_2\\
=&E_2^2\left(\int^{p_2}\!\!\!\int\!e^{-R_2}+L_1p_2+\widetilde{L}_1\right)+2p_3D_2\left(e^{-R_2}\right)+p_3^2\p_2\left(e^{-R_2}\right)+p_4e^{-R_2}\\
=&e^{-R_2}\left[e^{R_2}\left(E_2^2\int^{p_2}\!\!\!\int\!e^{-R_2}+E_2^2\left(L_1p_2+\widetilde{L}_1\right)\right)-2D_2R_2\cdot p_3-\p_2R_2\cdot p_3^2+p_4\right]\\
=&\rho_2(p_4-f_3).
\end{split}
\end{equation*}

(b) $\Rightarrow$ (c): Now, we extract $p_2$ from $E_2^2(L_1p_2+\widetilde{L}_1)$. Using Lemma \ref{2.2}(d), we consider, in a rather tricky way,
\begin{equation*}
\begin{split}
  E_2^2(L_1p_2+\widetilde{L}_1)
&=E_2^1(L_1p_2+\widetilde{L}_1)+D_2^2L_1\\
&=E_2^1(L_1p_2+\widetilde{L}_1)-E_2^1D_2\int^{p_1}\!L_1\\
&=E_2^1\left(\widetilde{L}_1-D_1\int^{p_1}\!L_1\right)\\
&=\big(\p_0-(D_1+p_2\p_1)\p_1\big)\left(\widetilde{L}_1-D_1\int^{p_1}\!L_1\right)\\
&=E_1^1\left(\widetilde{L}_1-D_1\int^{p_1}\!L_1\right)-p_2\p_1^2\left(\widetilde{L}_1-D_1\int^{p_1}\!L_1\right) \ .\\
\end{split}
\end{equation*}
Write $\displaystyle f_1=-\p_1^2\left(\widetilde{L}_1-D_1\int^{p_1}\!L_1\right)$, so that
\begin{equation}\label{thmApf:eq3}
\widetilde{L}_1 = -\int^{p_1} \!\!\! \int \! f_1 + D_1 \int^{p_1} \! L_1 + L_0 p_1 + \widetilde{L}_0
\end{equation}
for some functions $L_0(x,p_0)$ and $\widetilde{L}_0(x,p_0)$. It follows that
\begin{equation*}
\begin{split}
  E^2_2 \left( L_1 p_1 + \widetilde{L}_1 \right)
&=f_1 p_2 + E^1_1 \left(-\int^{p_1} \!\!\! \int \! f_1 + L_0 p_1 + \widetilde{L}_0 \right) \\
&=f_1 p_2 - E^1_1 \int^{p_1} \!\!\! \int \! f_1 + E^1_1 \left( L_0 p_1 + \widetilde{L}_0 \right) \ .\\
\end{split}
\end{equation*}
Putting these back into \eqref{3.2}, \eqref{3.3} follows.

(c) $\Rightarrow$ (b): When \eqref{3.3} holds, we can simply define $\widetilde{L}_1$ by \eqref{thmApf:eq3} to obtain \eqref{3.2}.

(c) $\Rightarrow$ (d): We will show that $E_1^1(L_0p_1+\widetilde{L}_0)$ is independent of $p_1$. Indeed,
\begin{eqnarray*}
E_1^1(L_0p_1+\widetilde{L}_0)& =& \p_0(L_0p_1+\widetilde{L}_0)-(\p_x+p_1\p_0)L_0\\
&=& \p_0\widetilde{L}_0-\p_xL_0.
\end{eqnarray*}
Setting $f_0=\p_0\widetilde{L}_0-\p_xL_0$, we obtain the formula for $f_3$ in \eqref{1.5}, and
$$\widetilde{L}_0=\int^{p_0}\!f_0+\p_x\int^{p_0}\!L_0+L_{-1}$$
for some function $L_{-1}(x)$, which implies
\begin{equation*}
\begin{split}
L_2
&=\int^{p_2}\!\!\!\int\!e^{-R_2}-\int^{p_1}\!\!\!\int\!f_1+D_2\int^{p_1}\!L_1+L_0p_1+\int^{p_0}\!f_0+\p_x\int^{p_0}\!L_0+L_{-1}\\
&=\int^{p_2}\!\!\!\int\!e^{-R_2}-\int^{p_1}\!\!\!\int\!f_1+\int^{p_0}\!f_0+D_2\left(\int^{p_1}\!L_1+\int^{p_0}\!L_0+\int^x\!L_{-1}\right).
\end{split}
\end{equation*}
This is precisely the expression for $L_2$ in (1.5) once we set $\displaystyle N_1=\int^{p_1}\!L_1+\int^{p_0}\!L_0+\int^x\!L_{-1}$.

(d) $\Rightarrow$ (c): We take $L_1=\p_1N_1$. Then there exists a function $N_0(x,p_0)$ such that
\begin{equation*}
N_1=\int^{p_1}\!L_1+N_0.
\end{equation*}
Take $L_0=\p_0N_0$ and $\displaystyle\widetilde{L}_0=\int^{p_0}\!f_0+\p_xN_0$. Then $f_0=\p_0\widetilde{L}_0-\p_xL_0$ and
\begin{equation*}
\begin{split}
  \int^{p_0}\!f_0+D_2N_1
&=\int^{p_0}\!f_0+D_2\int^{p_1}\!L_1+D_1N_0\\
&=\int^{p_0}\!f_0+D_2\int^{p_1}\!L_1+\p_0N_0\cdot p_1+\p_xN_0\\
&=D_2\int^{p_1}\!L_1+L_0p_1+\widetilde{L}_0\ ,
\end{split},
\end{equation*}
which implies \eqref{3.3}.

The proof of Theorem \ref{thmA} is completed.
\end{proof}

%
%

%
%
\section{Consistency with Fels' Conditions}

In \cite{F} it is shown that there are a non-degenerate Lagrangian $L_2(x,p_0,p_1,p_2)$ and a multiplier $\rho_3(x,p_0,\dots,p_3)$ such that $\rho_3(p_4-f_3)=E_4^2 L_2$ holds if and only if $f_3$ satisfies
\begin{align*}
T_5 :=\, & \dfrac{1}{6}\p_3^3 f_3=0, \\
I_1 :=\, & \p_1f_3+\dfrac{1}{2}\dd[2]{x}\p_3f_3-\dd{x}\p_2f_3-\dfrac{3}{4}\p_3f_3\cdot\dd{x}\p_3f_3+\dfrac{1}{2}\p_2f_3\cdot\p_3f_3+\dfrac{1}{8}(\p_3f_3)^3=0,
\end{align*}
where $d/dx = D_2 +f_3 \p_3$.  As pointed out before, whenever $\rho_3(p_4-f_3)=E_4^2 L_2$ holds for some $\rho_3$, $\rho_3$ is up to second order and we could replace it by $\rho_2$.  To show the consistency of our results with Fels' conditions, we need to verify that the solution given in \eqref{1.5} satisfies $T_5=I_1=0$.  As $f_3$ is quadratic in $p_3$, $T_5=0$ holds trivially.  It suffices to verify that $I_1$ vanishes.  By verifying this, we see that \eqref{1.5} essentially gives the general solution to the system of non-linear partial differential equations given by $T_5=I_1=0$.

%
%
\begin{prop}\label{4.1}
If
\begin{equation*}
f_3 = \p_2 R_2 \cdot p^2_3 + 2D_2 R_2 \cdot p_3 - e^{R_2} \left( E^2_2 \int^{p_2} \!\!\! \int \! e^{-R_2} + f_1 p_2 - E^1_1 \int^{p_1} \!\!\! \int \! f_1 + f_0 \right),
\end{equation*}
for some functions $R_2$, $f_1$ and $f_0$, then $I_1=0$.
\end{prop}
%
%

We begin with a lemma.
\begin{lem}[Preliminary calculations]\label{4.2}
We have
\begin{enumerate}[(a)]
\item $\p_2^2E_2^2=(D_2^2\p_2+3D_2\p_1+3\p_0)\p_2^2$,
\item $\p_1E_1^1=-D_1\p_1^2$,
\item $(D_2\p_2-\p_1)E^2_2=D^3_2\p^2_2$.
\end{enumerate}
\end{lem}

\begin{remark}
The bracket in (a) is apparently of binomial type, and that in (c) resembles an Euler--Lagrange type operator.
\end{remark}

\begin{proof}
\begin{enumerate}[(a)]
\item
We have, by Lemma \ref{2.2}(b),
\begin{equation*}
\begin{split}
\p_2 E^2_2
=\, & \p_2 ( \p_0 - D_2 \p_1 + D^2_2 \p_2 ) \\
=\, & \p_0 \p_2 - D_2 \p_1 \p_2 - \p^2_1 + D^2_2 \p^2_2 + 2D_2 \p_1 \p_2 + \p_0 \p_2 \\
=\, & D^2_2 \p^2_2 + D_2 \p_1 \p_2 + 2\p_0 \p_2 - \p^2_1 \\
\p^2_2 E^2_2,
=\, & D^2_2 \p^3_2 + 2D_2 \p_1 \p^2_2 + \p_0 \p^2_2 + D_2 \p_1 \p^2_2 + \p^2_1 \p_2 +2\p_0 \p^2_2 - \p^2_1 \p_2 \\
=\, & D^2_2 \p^3_2 + 3D_2 \p_1 \p^2_2 + 3\p_0 \p^2_2 \\
=\, & (D_2^2\p_2+3D_2\p_1+3\p_0)\p_2^2.
\end{split}
\end{equation*}
We simply have
\begin{equation*}\begin{split}
\p_1E_1^1
&=\p_1(\p_0-D_1\p_1)\\
&=\p_1\p_0-(D_1\p_1+\p_0)\p_1\\
&=-D_1\p_1^2.
\end{split}\end{equation*}
\item
On one hand we have  $D_2\p_2E^2_2=D_2(D^2_2\p^2_2+D_2\p_1\p_2+2\p_0\p_2-\p^2_1)$, by the calculations in the proof of part (b). On the other hand, by Lemma \ref{2.3},
\begin{align*}
\p_1 E^2_2 =\, & \p_1 ( \p_0 - D_2 \p_1 + D^2_2 \p_2 ) \\
           =\, & \p_1 \p_0 - D_2 \p^2_1 - \p_0 \p_1 + D^2_2 \p_1 \p_2 + 2D_2 \p_0 \p_2 \\
           =\, & D^2_2 \p_1 \p_2 + 2D_2 \p_0 \p_2 - D_2 \p^2_1,
\end{align*}

\end{enumerate}
\end{proof}
%
%

%
%
\begin{proof}[Proof of Proposition \ref{4.1}]
By direct computations, if we write $f_3 = Ap^2_3 + 2Bp_3 + C$, where $A=\p_2R_2, B=D_2R_2$ and $C$ satisfies
\begin{equation}\label{thmC:eq3}
-e^{-R_2}C = E^2_2 \int^{p_2} \!\!\! \int \! e^{-R_2} + f_1p_2 - E^1_1 \int^{p_1} \!\!\! \int \! f_1 + f_0,
\end{equation}
we have
\begin{align*}
    & \p_1 f_3 - \dd{x} \p_2 f_3 + \dd[2]{x} \left( \dfrac{\p_3 f_3}{2} \right) + \left( \dfrac{\p_3 f_3}{2} \right) \left( \p_2 f_3 - 3\dd{x} \left( \dfrac{\p_3 f_3}{2} \right) \right) + \left( \dfrac{\p_3 f_3}{2} \right)^3 \\
=\, & \left[\p_2 (D_2 A) + \p_1 A - \p^2_2 B\right] p_3^2 \\
    & +\left[-\p^2_2 C + 2A \p_2 C + ( \p_2 A - A^2 ) C\right. \\
    & \left.+D^2_2 A + BD_2 A-A(D_2B-B^2) - 3B \p_2 B +3 \p_1 B\right] p_3 \\
    & +\left[ B ( \p_2 C - AC ) - D_2 (\p_2 C - AC ) + (\p_1 C + (D_2 A - \p_2 B)C) \right. \\
    & \left. +D^2_2 B - 3B D_2 B + B^3\right]
\end{align*}
First of all, the coefficient of $p_3^2$ is
\begin{equation*}
\p_2D_2\p_2R_2+\p_1\p_2R_2-\p^2_2D_2R_2=\p_2(D_2\p_2+\p_1-\p_2D_2)R_2=0,
\end{equation*}
which vanishes since the last bracket does. Therefore, it suffices to show that both the coefficient of $p_3$ and the zero order term are identically zero, that is,
\begin{equation}\label{thmC:eq1}
\begin{split}
-\p^2_2 C + 2\p_2 R_2 \cdot \p_2 C + (\p^2_2 R_2 - (\p_2 R_2)^2) C + D^2_2 \p_2 R_2 + D_2 R_2 \cdot D_2 \p_2 R_2 & \\
 + \p_2 R_2(D^2_2R_2-(D_2 R_2)^2) - 3D_2 R_2 \cdot \p_2 (D_2 R_2) + 3\p_1 D_2 R_2 & = 0
\end{split}
\end{equation}
and
\begin{equation}\label{thmC:eq2}
\begin{split}
D_2 R_2 ( \p_2 C - \p_2 R_2 \cdot C) - D_2 ( \p_2 C - \p_2 R_2 \cdot C )+(\p_1C-\p_1R_2\cdot C) & \\
 + D^3_2 R_2 - 3D_2R_2 \cdot D^2_2 R_2 + (D_2R_2)^3 & = 0 \ .
\end{split}
\end{equation}

To verify \eqref{thmC:eq1}, we differentiate \eqref{thmC:eq3} with respect to $p_2$ once and twice to obtain
\begin{equation}\label{thmC:eq4}
-e^{-R_2} \p_2 C + e^{-R_2} \p_2 R_2 \cdot C = \p_2 E^2_2 \int^{p_2} \!\!\! \int \! e^{-R_2} + f_1,
\end{equation}
and
\begin{equation*}
-e^{-R_2} \p^2_2 C + 2e^{-R_2} \p_2 R_2 \cdot \p_2 C + e^{-R_2} \p^2_2 R_2 \cdot C - e^{-R_2} ( \p_2 R_2)^2 C = \p^2_2 E^2_2 \int^{p_2} \!\!\! \int \! e^{-R_2}
\end{equation*}
respectively. Factoring out $e^{-R_2}$ in the left hand side of the last equality, and applying Lemma \ref{4.2}(a) to the right hand side, we have
\begin{align*}
    & e^{-R_2} (-\p^2_2 C + 2\p_2 R_2 \cdot \p_2 C + ( \p^2_2 R_2 - ( \p_2 R_2 )^2 ) C) \\
=\, & (D^2_2 \p_2 + 3D_2 \p_1 + 3\p_0 ) e^{-R_2}  \\
=\, & D^2_2 (-e^{-R_2} \p_2 R_2) + 3D_2 (-e^{-R_2} \p_1 R_2 ) +3 ( -e^{-R_2} \p_0 R_2) \\
=\, & D_2 (-e^{-R_2} D_2 \p_2 R_2 + e^{-R_2} D_2 R_2 \cdot \p_2 R_2 ) \\
    & +3(-e^{-R_2}D_2\p_1R_2+e^{-R_2}D_2R_2\cdot \p_1R_2) - 3e^{-R_2}\p_0R_2 \\
=\, & -e^{-R_2} D^2_2 \p_2 R_2 + e^{-R_2} D_2 R_2 \cdot D_2 \p_2 R_2 \\
    & - e^{-R_2}(D_2R_2)^2 \p_2 R_2 + e^{-R_2} D^2_2 R_2 \cdot \p_2 R_2 + e^{-R_2} D_2R_2 \cdot D_2 \p_2 R_2 \\
    & -3e^{-R_2} D_2 \p_1 R_2 + 3e^{-R_2}D_2 R_2 \cdot \p_1 R_2 - 3e^{-R_2} \p_0 R_2 \\
=\, & -e^{-R_2}(D^2_2\p_2R_2-2D_2R_2\cdot D_2\p_2R_2-\p_2R_2(D_2^2R_2-(D_2R_2)^2) \\
    & +3\p_1D_2R_2-3D_2R_2\cdot\p_1R_2) \\
=\, & -e^{-R_2}(D^2_2\p_2R_2+D_2R_2\cdot D_2\p_2R_2-\p_2R_2(D^2_2R_2-(D_2R_2)^2) \\
    & -3D_2R_2\cdot\p_2D_2R_2+3\p_1D_2R_2),
\end{align*}
from which \eqref{thmC:eq1} follows.

To verify  \eqref{thmC:eq2}, we recall \eqref{thmC:eq4} and compute
\begin{align*}
D_2 R_2 (\p_2 C - \p_2 R_2 \cdot C ) =\, & -D_2 R_2 \cdot e^{R_2} \left( \p_2 E^2_2 \int^{p_2} \!\!\! \int \! e^{-R_2} + f_1 \right) \\
=\, & -(D_2 e^{R_2} ) \left( \p_2 E^2_2 \int^{p_2} \!\!\! \int \! e^{-R_2} + f_1 \right),
\end{align*}
\begin{align*}
-D_2 (\p_2 C - \p_2 R_2 \cdot C) =\, & D_2 \left(e^{R_2} \left( \p_2 E^2_2 \int^{p_2} \!\!\! \int \! e^{-R_2} + f_1 \right)\right) \\
=\, & (D_2 e^{R_2} ) \left(\p_2 E^2_2 \int^{p_2} \!\!\! \int \! e^{-R_2} + f_1 \right) \\
& + e^{R_2} D_2 \left(\p_2 E^2_2 \int^{p_2} \!\!\! \int \! e^{-R_2} + f_1\right),
\end{align*}
as well as
\begin{multline*}
-e^{-R_2}(\p_1C-\p_1R_2\cdot C)
= \p_1 E^2_2 \int^{p_2} \!\!\! \int \! e^{-R_2} + \p_1 f_1 \cdot p_2 - \p_1 E^1_1 \int^{p_1} \!\!\! \int \! f_1 \\
= \p_1 E^2_2 \int^{p_2} \!\!\! \int \! e^{-R_2} + \p_1 f_1 \cdot p_2 + D_1 f_1.
\quad\text{(by Lemma \ref{4.2}(b))}
\end{multline*}
We rewrite the last formula as
 $$\p_1C-\p_1R_2\cdot C = -e^{R_2} \left( \p_1 E^2_2 \int^{p_2} \!\!\! \int \! e^{-R_2} + D_2 f_1\right).$$

Now, putting things together, we have
\begin{align*}
& D_2 R_2 (\p_2 C - \p_2 R_2 \cdot C ) - D_2 ( \p_2 C - \p_2 R_2 \cdot C ) + (\p_1 C - \p_1 R_2 \cdot C ) \\
=\, & e^{R_2} (D_2 \p_2 - \p_1 ) E^2_2 \int^{p_2} \!\!\! \int \! e^{-R_2} \\
=\, & e^{R_2} D^3_2 e^{-R_2} \quad \text{(by Lemma \ref{4.2}(c))}\\
=\, & e^{R_2} D^2_2 (-e^{-R_2} D_2R_2) \\
=\, & e^{R_2} D_2 (-e^{-R_2} D^2_2R_2 + e^{-R_2} (D_2R_2)^2) \\
=\, & e^{R_2} (-e^{R_2} D^3_2 R_2 + e^{-R_2} D_2R_2 \cdot D^2_2 R_2 + 2e^{-R_2} D_2 R_2 \cdot D^2_2 R_2 - e^{-R_2} (D_2 R_2)^3) \\
=\, & -D^3_2 R_2 + 3D_2 R_2 \cdot D^2_2 R_2 - (D_2R_2)^3\ ,
\end{align*}
and \eqref{thmC:eq2} follows.

We have, therefore, verified that our solution \eqref{1.5} fulfills Fels' conditions.
\end{proof}

%
%
\section{The Higher Order Equation}
A comparison between Theorems \ref{thmA} and \ref{thmB} reveals that \eqref{1.6} can be seen as a direct generalization of \eqref{1.5} except that $f_3$ is quadratic in $p_3$ whereas for $n\geq 3$, $f_{2n-1}$ is linear in $p_{2n-1}$. This results from the simple fact that $2n-1=n+1$ if and only if $n=2$. One may compare \eqref{thmApf:compare} and \eqref{thmBpf:compare}.

The proof for Theorem \ref{thmB} differs from that of Theorem \ref{thmA} in two ways. Firstly, we need a systematic way to extract the highest order terms, that is, to reduce the value of $m$ in an expression involving $E_m^n$. The reduction of $m$ in the powers of $D_m$ is contained in Lemma \ref{2.4}. The second point concerns the steps of reducing $m$ and $n$ in $E_m^n$. For $n=2$, the reduction can be symbolized as $E_4^2 \to E_2^2$ while that for $n\geq 3$ is $E_{2n}^n \to E_{2n-2}^n \to E_{2n-4}^{n-1} \to \cdots \to E_4^3 \to E_2^2$. The inductive step, which reduces $m$ by $2$ and $n$ by $1$, is contained in Lemma \ref{5.1}.

%
%
We remind the readers that the subscript notation for functions is explained in Section \ref{notations}.
\begin{lem}\label{5.1}
Suppose $n\geq 3$. Given $L_{n-1}$ and $\widetilde{L}_{n-1}$, there exist $f_{n-1}$, $L_{n-2}$ and $\widetilde{L}_{n-2}$ such that
\begin{align}\label{lem:E_{2n-2}^n,E_{2n-4}^{n-1}:eq1}
E^n_{2n-2}\left(L_{n-1}p_n + \widetilde{L}_{n-1} \right) =\, & f_{n-1}p_{2n-2} + (-1)^{n-1}E^{n-1}_{2n-3} \int^{p_{n-1}} \!\!\!\! \int \! f_{n-1}  \\
& + E^{n-1}_{2n-4} \left(L_{n-2}p_{n-1}+\widetilde{L}_{n-2} \right). \nonumber
\end{align}
In fact,
\begin{equation*}
f_{n-1}=(-1)^{n-1} \p^2_{n-1} \left(\widetilde{L}_{n-1} - D_{n-1} \int^{p_{n-1}} \! L_{n-1} \right).
\end{equation*}

Conversely, given $f_{n-1}$, $L_{n-1}$, $L_{n-2}$ and $\widetilde{L}_{n-2}$, there exists $\widetilde{L}_{n-1}$ such that \eqref{lem:E_{2n-2}^n,E_{2n-4}^{n-1}:eq1} holds.
\end{lem}
%
%

%
%
\begin{proof}
Consider
\begin{align*}
    & \left(E^n_{2n-2} - E^n_{2n-3}\right) \left(L_{n-1} p_n + \widetilde{L}_{n-1} \right) \\
=\, & \sum_{k=0}^n (-1)^k \left(D^k_{2n-2} - D^k_{2n-3}\right) \p_k \left(L_{n-1}p_n - \widetilde{L}_{n-1} \right) \\
=\, & \sum_{k=1}^n (-1)^k \sum_{0 < \norm{I} \leq k} a_{I}^{(k)} p_{2n-2}^{\abs{I}} D^{k-\norm{I}}_{2n-3} \p^{I} \p_k \left(L_{n-1}p_n + \widetilde{L}_{n-1} \right).
\end{align*}
We emphasize that in the last line, the multi-index $I$ depends on $k$ and has $(2n-2)$ elements, $I=I_{2n-2}^k=(i_{2n-3}^k,i_{2n-4}^k,\dots,i_0^k)$. However, in the following we will suppress the dependence for simplicity and consider
its elements
$$I=(i_{2n-3},\dots,i_{n},i_{n-1},i_{n-2},\dots,i_0).$$

In fact, only a few terms in the summations remain. To see this, since the function $L_{n-1}p_n+\widetilde{L}_{n-1}$ depends only up to $p_n$, we may assume $i_{2n-3}=\cdots=i_{n+1}=0$. Moreover, we also have $i_n=0$ in the case $k=n$, because the function to be differentiated is linear in $p_n$. Recalling the definition of $\norm{I}$, the condition $0<\norm{I}\leq{k}$ is actually highly restrictive:
\begin{equation*}\begin{cases}
0<(n-1)i_{n-1}+ni_{n-2}+(n+1)i_{n-3}+\cdots\leq{n},&\text{if }k=n,\\
0<(n-2)i_n+(n-1)i_{n-1}+ni_{n-2}+\cdots\leq{k}\leq{n-1},&\text{if }k\leq{n-1}.
\end{cases}\end{equation*}
Clearly, $i_{n-3}=\cdots=i_0=0$ also holds true. One observes that it has exactly five solutions in $(k,i_n,i_{n-1},i_{n-2})$, as tabulated below.

\begin{center}\begin{tabular}{|c|c|c|c|l|}
\hline
$k$   & $i_{n}$ & $i_{n-1}$ & $i_{n-2}$ & \qquad\quad$a_I^{(k)}$\\
\hline\hline
$n$   & 0       & 0         & 1         & \quad\ \,$\frac{n!}{n!1!0!}=1$ \\\hline
$n$   & 0       & 1         & 0         & $\frac{n!}{(n-1)!1!1!}=n$ \\\hline
$n-1$ & 0       & 1         & 0         & $\frac{(n-1)!}{(n-1)!1!0!}=1$ \\\hline
$n-1$ & 1       & 0         & 0         & $\frac{(n-1)!}{(n-2)!1!1!}=n-1$ \\\hline
$n-2$ & 1       & 0         & 0         & $\frac{(n-2)!}{(n-2)!1!0!}=1$ \\
\hline
\end{tabular}\end{center}
Hence, the difference we consider above becomes
\begin{equation}\label{5.2}\begin{split}
&\quad\,(E_{2n-2}^n-E_{2n-3}^n)\left(L_{n-1}p_n+\widetilde{L}_{n-1}\right)\\
&=p_{2n-2}\bigg((-1)^n(\p_{n-2}L_{n-1}+nD_{2n-3}\p_{n-1}L_{n-1})\\
&\qquad\qquad+(-1)^{n-1}\left(\p_{n-1}^2(L_{n-1}p_n+\widetilde{L}_{n-1})+(n-1)D_{2n-3}\p_{n-1}L_{n-1}\right)\\
&\qquad\qquad+(-1)^{n-2}\p_{n-2}L_{n-1}\bigg)\\
&=p_{2n-2}(-1)^{n-1}\left(\p_{n-1}^2(L_{n-1}p_n+\widetilde{L}_{n-1})-D_{2n-3}\p_{n-1}L_{n-1}-2\p_{n-2}L_{n-1}\right)\\
&=p_{2n-2}(-1)^{n-1}\left(\p_{n-1}^2\widetilde{L}_{n-1}-(D_{n-1}\p_{n-1}+2\p_{n-2})L_{n-1}\right)\\
&=p_{2n-2}(-1)^{n-1}\p_{n-1}^2\left(\widetilde{L}_{n-1}-D_{n-1}\int^{p_{n-1}}\!L_{n-1}\right).\quad\text{(by Lemma \ref{2.3b})}
\end{split}\end{equation}

Write $ \displaystyle f_{n-1} = (-1)^{n-1} \p^2_{n-1} \left( \widetilde{L}_{n-1} - D_{n-1} \int^{p_{n-1}} \! L_{n-1} \right)$ so that
\begin{equation}\label{lem:E_{2n-2}^n,E_{2n-4}^{n-1}:eq2}
\widetilde{L}_{n-1} = (-1)^{n-1} \int^{p_{n-1}} \!\!\!\! \int \! f_{n-1} + D_{n-1} \int^{p_{n-1}} \! L_{n-1} + L_{n-2} p_{n-1} + \widetilde{L}_{n-2},
\end{equation}
for some functions $L_{n-2}$ and $\widetilde{L}_{n-2}$. Using Lemma \ref{2.5}, we have
\begin{equation*}
\begin{split}
E_{2n-3}^n\left(L_{n-1}p_n+D_{n-1}\int^{p_{n-1}}\!L_{n-1}\right)
&=E_{2n-3}^nD_n\int^{p_{n-1}}\!L_{n-1}\\
&=E_{2n-3}^nD_{2n-3}\int^{p_{n-1}}\!L_{n-1}\\
&=(-1)^nD_{2n-3}^{n+1}\p_n\int^{p_{n-1}}\!L_{n-1}\\
&=0.
\end{split}
\end{equation*}
Therefore,
\begin{align*}
    & E^n_{2n-2}\left(L_{n-1}p_n + \widetilde{L}_{n-1} \right) \\
=\, & \left(E^n_{2n-2} - E^n_{2n-3} \right) \left(L_{n-1} p_n + \widetilde{L}_{n-1} \right)+ E^n_{2n-3} \left( L_{n-1} p_n + \widetilde{L}_{n-1} \right) \\
=\, & f_{n-1} p_{2n-2} + E^n_{2n-3} \left( (-1)^{n-1} \int^{p_{n-1}} \!\!\!\! \int \! f_{n-1} + D_n \int^{p_{n-1}} \! L_{n-1} + L_{n-2} p_{n-1} + \widetilde{L}_{n-1} \right) \\
=\, & f_{n-1}p_{2n-2} + (-1)^{n-1}E^{n-1}_{2n-3} \int^{p_{n-1}} \!\!\!\! \int \! f_{n-1} + E^{n-1}_{2n-3} \left(L_{n-2}p_{n-1}+\widetilde{L}_{n-1} \right),
\end{align*}

We expand similarly
\begin{align*}
& \left(E^{n-1}_{2n-3} - E^{n-1}_{2n-4}\right) \left(L_{n-2}p_{n-1} + \widetilde{L}_{n-2} \right)  \\
=\, & \sum_{k=1}^{n-1} (-1)^k \sum_{0< \norm{I} \leq k} a^{(k)}_{I} p^{\abs{I}}_{2n-3} D^{k-\norm{I}}_{2n-4} \p^{I} \p_k \left(L_{n-2}p_{n-1} + \widetilde{L}_{n-2} \right).
\end{align*}
Proceeding as above, we arrive at the inequalities
\begin{equation*}\begin{cases}
0<(n-2)i_{n-1}+(n-1)i_{n-2}+\cdots\leq{k}\leq{n-2},&\text{if }k\leq{n-2},\\
0<(n-1)i_{n-2}+ni_{n-3}+\cdots\leq{n-1},&\text{if }k=n-1.
\end{cases}\end{equation*}
The only two solutions are $(k,i_{n-1},i_{n-2})=(n-2,1,0)$ and $(n-1,0,1)$.
By symmetry and the alternating sign $(-1)^k$, these non-zero terms cancel each other and we obtain \ref{5.1}.


Conversely, given any $f_{n-1}$, $L_{n-1}$, $L_{n-2}$ and $\widetilde{L}_{n-2}$, we define $\widetilde{L}_{n-1}$ by \eqref{lem:E_{2n-2}^n,E_{2n-4}^{n-1}:eq2}. The above calculations show that \eqref{lem:E_{2n-2}^n,E_{2n-4}^{n-1}:eq1} holds.
\end{proof}

\medskip

\begin{remark}
The computations in equation \eqref{5.2} can be done alternatively by inserting $p_{2n-2}\p_{2n-2}$ in the front and using Lemma \ref{2.3}. In this way, Lemma \ref{2.4} is not required in its full strength (regarding the explicit coefficient $a_{I}^{(k)}$).
\end{remark}

\medskip

%
%

We can now prove Theorem \ref{thmB} in three steps. First, we compute the difference $(E_{2n}^n-E_{2n-2}^n)L_n$ and decouple the equation. Next, by a repeated application of Lemma \ref{5.1}, we reduce $E_{2n-2}^n$ to $E_2^2$. Finally, the proof of Theorem \ref{thmA} ((b) $\Leftrightarrow$ (d)) applies and we obtain the desired formulas.

\medskip

%
%
\begin{proof}[Proof of Theorem \ref{thmB}]
Similar to the proof of Theorem \ref{thmA}, we will show the equivalence of the following statements, where $n\geq3$.
\begin{enumerate}[(a)]
\item There exist $\rho_n$, $f_{2n-1}$ and $L_n$ such that \eqref{1.4} holds, that is,
    \begin{equation*}
    \rho_n(p_{2n}-f_{2n-1})=E_{2n}^nL_n.
    \end{equation*}
\item There exist $R_n$, $L_{n-1}$ and $\widetilde{L}_{n-1}$ such that
    \begin{equation}\label{thmBpf:eq1}
    \left\{
    \begin{array}{rcl}
      \rho_n & = & e^{-R_n}\ , \\
    f_{2n-1} & = & nD_{n+1}R_n\cdot p_{2n-1} \\
             &   & \displaystyle -e^{R_n}\left[(-1)^nE_{2n-2}^n\int^{p_n}\!\!\!\int\!e^{-R_n}+E_{2n-2}^n\left(L_{n-1}p_n+\widetilde{L}_{n-1}\right)\right]\ ,\\
         L_n & = & \displaystyle (-1)^n\int^{p_n}\!\!\!\int\!e^{-R_n}+L_{n-1}p_n+\widetilde{L}_{n-1}  \ .
    \end{array}
    \right.
    \end{equation}
\item There exist $R_n$, $(f_m)_{2\leq{m}\leq{n-1}}$, $(L_m)_{1\leq{m}\leq{n-1}}$ and $\widetilde{L}_1$ such that
    \begin{equation}\label{thmBpf:eq2}
    \left\{
    \begin{array}{rcl}
      \rho_n & = & e^{-R_n}\ , \\
    f_{2n-1} & = & \displaystyle nD_{n+1}R_n\cdot p_{2n-1}-e^{R_n}\left[(-1)^nE_{2n-2}^n\int^{p_n}\!\!\!\int\!e^{-R_n}\right. \\
             &   & \displaystyle +\left.\sum_{m=2}^{n-1}\left(f_mp_{2m}+(-1)^mE_{2m-1}^m\int^{p_m}\!\!\!\int\!f_m\right)+E_2^2\left(L_1p_2+\widetilde{L}_1\right)\right]\ ,\\
         L_n & = & \displaystyle \int^{p_n}\!\!\!\int\!e^{-R_n}+\sum_{m=2}^{n-1}(-1)^m\int^{p_m}\!\!\!\int\!f_m+D_n\sum_{m=2}^{n-1}\int^{p_m}\!L_m+L_1p_2+\widetilde{L}_1 \ .
    \end{array}
    \right.
    \end{equation}
\item There exist $R_n$, $(f_m)_{0\leq{m}\leq{n-1}}$ and $N_{n-1}$ such that \eqref{1.6} holds, that is,
    \begin{equation*}
    \left\{
    \begin{array}{rcl}
      \rho_n & = & e^{-R_n}\ , \\
    f_{2n-1} & = & \displaystyle nD_{n+1}R_n\cdot p_{2n-1}-e^{R_n}\left[(-1)^nE_{2n-2}^n\int^{p_n}\!\!\!\int\!e^{-R_n}\right. \\
             &   & \displaystyle +\left.\sum_{m=1}^{n-1}\left(f_mp_{2m}+(-1)^mE_{2m-1}^m\int^{p_m}\!\!\!\int\!f_m\right)+f_0\right]\ ,\\
         L_n & = & \displaystyle \int^{p_n}\!\!\!\int\!e^{-R_n}+\sum_{m=1}^{n-1}(-1)^m\int^{p_m}\!\!\!\int\!f_m+\int^{p_0}\!f_0+D_nN_{n-1} \ .
    \end{array}
    \right.
    \end{equation*}
\end{enumerate}
Note that the equivalence of (c) and (d) corresponds to that of (b) and (d) in the proof of Theorem \ref{thmA}.

(a) $\Rightarrow$ (b): Suppose \eqref{1.4} has a solution. In the same spirit of the proof of Lemma \ref{5.1}, we expand
\begin{align*}
\left(E_{2n}^n-E_{2n-1}^n\right)L_n =\, & \sum_{k=0}^n (-1)^k \left(D_{2n}^k-D_{2n-1}^k\right)\p_k L_n \\
                                    =\, & \sum_{k=1}^n (-1)^k \sum_{0 < \norm{I} \leq k} a^{(k)}_{I} p^{\abs{I}}_{2n} D^{k- \norm{I}}_{2n-1} \p^{I} \p_k L_n.
\end{align*}
The inequality
\begin{equation*}
0<ni_n+(n+1)i_{n-1}+\cdots\leq{k}\leq{n}
\end{equation*}
has a unique solution $(k,i_n)=(n,1)$ with all other $i_j$'s vanishing. The coefficient is $a_{(0,\dots,1,\dots,0)}^{(n)}=n!/(n!1!0!)=1$. Hence,
$$(E_{2n}^n-E_{2n-1}^n)L_n=(-1)^np_{2n}\p_n^2L_n.$$
Similarly, to find the difference
$$(E_{2n-1}^n-E_{2n-2}^n)L_n=\sum_{k=1}^{n}(-1)^k\sum_{0<\norm{I}\leq{k}}a_{I}^{(k)}p_{2n-1}^{\abs{I}}D_{2n-2}^{k-\norm{I}}\p^{I}\p_{k}L_n,$$
we consider
\begin{equation*}\begin{cases}
0<(n-1)i_n+ni_{n-1}+\cdots\leq{n},&\text{if }k=n,\\
0<(n-1)i_n+ni_{n-1}+\cdots\leq{k}\leq{n-1},&\text{if }k\leq{n-1}.
\end{cases}\end{equation*}
Of the three solutions $(k,i_n,i_{n-1})=(n,1,0),\,(n,0,1)$ and $(n-1,1,0)$, the last two yield cancelling terms. This gives
$$(E_{2n-1}^n-E_{2n-2}^n)L_n=(-1)^{n}np_{2n-1}D_{2n-1}\p_n^2L_n.$$


Therefore, \eqref{1.4} can be rewritten as
$$\rho_n (p_{2n} - f_{2n-1}) = (-1)^n p_{2n} \p_n^2 L_n + (-1)^n n p_{2n-1} D_{2n-1} \p_n^2 L_n + E^n_{2n-2} L_n,$$
or
\begin{equation*}\begin{cases}
          \rho_n = (-1)^n \p_n^2 L_n \ ,\\
-\rho_n f_{2n-1} = (-1)^n n p_{2n-1} D_{2n-1} \p_n^2 L_n + E^n_{2n-2} L_n \ .
\end{cases}\end{equation*}
Substituting the first equation into the second and differentiating both sides with respect to $p_{2n-1}$, we have
\begin{equation*}
-\rho_n \p_{2n-1} f_{2n-1} = n D_{2n-2} \rho_n = n D_{n+1} \rho_n,
\end{equation*}
or
\begin{equation*}
D_{n+1} \rho_n + \left( \dfrac{1}{n} \p_{2n-1} f_{2n-1} \right)\rho_n = 0.
\end{equation*}
In particular, we infer that $\p_{2n-1}f_{2n-1}$ depends on at most $p_{n+1}$. By Lemma \ref{2.6}, this equation is solvable if and only if there exists $R_n$ such that $\rho_n = e^{-R_n}$ and
\begin{equation}\label{thmBpf:compare}
\dfrac{1}{n}\p_{2n-1} f_{2n-1} = D_{n+1} R_n.
\end{equation}
Hence,
\begin{equation*}
f_{2n-1} = n D_{n+1} R_n \cdot p_{2n-1} - e^{R_n} E^n_{2n-2} L_n.
\end{equation*}
From $e^{-R_n} = (-1)^n \p_n^2 L_n$, there exist functions $L_{n-1}$ and $\widetilde{L}_{n-1}$ such that
\begin{equation*}
L_n = (-1)^n \int^{p_n} \!\!\! \int \! e^{-R_n} + L_{n-1} p_n + \widetilde{L}_{n-1},
\end{equation*}
and so
\begin{equation*}
f_{2n-1} = n D_{n+1} R_n \cdot p_{2n-1} - e^{R_n} \left[(-1)^n E^n_{2n-2} \int^{p_n} \!\!\! \int \! e^{-R_n} + E^n_{2n-2} \left( L_{n-1} p_n + \widetilde{L}_{n-1} \right) \right].
\end{equation*}
We have proved that  \eqref{thmBpf:eq1}  holds.

(b) $\Rightarrow$ (a): Suppose \eqref{thmBpf:eq1} holds. By the above calculations, we have
\begin{equation*}
\begin{split}
 &E_{2n}^nL_n\\
=&(-1)^np_{2n}\p_n^2L_n+(-1)^nnp_{2n-1}D_{2n-1}\p_n^2L_n+E_{2n-2}^nL_n\\
=&e^{-R_n}p_{2n}+nD_{2n-1}\left(e^{-R_n}\right)p_{2n-1}+E_{2n-2}^n\left((-1)^n\int^{p_n}\!\!\!\int\!e^{-R_n}+L_{n-1}p_n+\widetilde{L}_{n-1}\right)\\
=&e^{-R_n}\bigg[p_{2n}-nD_{n+1}R_n\cdot p_{2n-1}\\
 &\left.+e^{R_n}\left((-1)^nE_{2n-2}^n\int^{p_n}\!\!\!\int\!e^{-R_n}+E_{2n-2}^n\left(L_{n-1}p_n+\widetilde{L}_{n-1}\right)\right)\right]\\
=&\rho_n(p_{2n}-f_{2n-1}).
\end{split}
\end{equation*}

(b) $\Leftrightarrow$ (c): We wish to show, by induction, that \eqref{thmBpf:eq1} is equivalent to the following statement for any $\ell$ with $2\leq\ell\leq{n}$ (call it $P(\ell)$): there exist functions $(f_m)_{\ell\leq{m}\leq{n-1}}$, $(L_m)_{\ell-1\leq{m}\leq{n-1}}$ and $\widetilde{L}_{\ell-1}$ such that
\begin{equation*}\begin{cases}
f_{2n-1}&=\,
 \displaystyle n D_{n+1} R_n \cdot p_{2n-1} - e^{R_n} \left[ (-1)^n E^n_{2n-2} \int^{p_n} \!\!\! \int \! e^{-R_n} \right.\\
& \quad\displaystyle \left. +\sum^{n-1}_{m=\ell} \left( f_m p_{2m} + (-1)^m E^m_{2m-1} \int^{p_m} \!\!\! \int \! f_m \right) + E^{\ell}_{2\ell -2} \left( L_{\ell-1} p_{\ell} + \widetilde{L}_{n-1} \right) \right] \ ,\\
L_n &=\,
 \displaystyle (-1)^n \int^{p_n} \!\!\! \int \! e^{-R_n} + \sum^{n-1}_{m=\ell} (-1)^m \int^{p_m} \!\!\! \int \! f_m + D_n \sum^{n-1}_{m=\ell} \int^{p_m} \! L_m \\
& \quad+ L_{\ell-1} p_{\ell} + \widetilde{L}_{\ell-1}   \ .
\end{cases}\end{equation*}
Clearly, $P(n)$ is vacuously true. Assume that $P(\ell)$ is true, where $3\leq\ell\leq{n}$. Then by Lemma \ref{5.1}, there exist functions $f_{\ell-1}$, $L_{\ell-2}$ and $\widetilde{L}_{\ell-2}$ such that
\begin{equation*}\begin{cases}
f_{2n-1}&=\,
 \displaystyle n D_{n+1} R_n \cdot p_{2n-1} - e^{R_n} \left[ (-1)^n E^n_{2n-2} \int^{p_n} \!\!\! \int \! e^{-R_n} \right.\\
&\quad \displaystyle \left. + \sum^{n-1}_{m=\ell} \left( f_m p_{2m} + (-1)^m E^m_{2m-1} \int^{p_m} \!\!\! \int \! f_m \right) \right. \\
&\quad \displaystyle \left. + f_{\ell-1} p_{2\ell-2} + (-1)^{\ell-1} E^{\ell-1}_{2\ell-3} \int^{p_{\ell-1}} \!\!\!\! \int \! f_{\ell-1} + E^{\ell-1}_{2\ell -4} \left( L_{\ell-2} p_{\ell-1} + \widetilde{L}_{n-2} \right) \right] \ ,\\
L_n&=\,
 \displaystyle (-1)^n \int^{p_n} \!\!\! \int \! e^{-R_n} + \sum^{n-1}_{m=\ell} (-1)^m \int^{p_m} \!\!\! \int \! f_m + D_n \sum^{n-1}_{m=\ell} \int^{p_m} \! L_m \\
& \quad\displaystyle + (-1)^{\ell-1} \int^{p_{\ell-1}} \!\!\!\! \int \! f_{\ell-1} + D_{\ell} \int^{p_{\ell-1}} \! L_{\ell-1} + L_{\ell-2} p_{\ell-1} + \widetilde{L}_{\ell-2} \ ,
\end{cases}\end{equation*}
which shows that the $P(\ell-1)$ is also true. On the other hand, if $P(\ell-1)$ is true, then the above formulas and the converse part of Lemma \ref{5.1} imply that $P(\ell)$ is true. Therefore, the statements $P(n),P(n-1),\dots,P(2)$ are all equivalent.

(c) $\Rightarrow$ (d): By exactly the same calculations as in the implications (b) $\Rightarrow$ (c) $\Rightarrow$ (d) of the proof of Theorem \ref{thmA}, there exist functions $f_1$, $f_0$, $L_0$ and $L_{-1}=L_{-1}(x)$ such that
\begin{equation*}\begin{cases}
E^2_2 \left( L_1 p_2 +\widetilde{L}_1 \right) &= f_1 p_2 - E^1_1 \displaystyle \int^{p_1} \!\!\! \int \! f_1 + f_0 \ ,\\
\displaystyle \widetilde{L}_1 &= - \displaystyle \int^{p_1} \!\!\! \int \! f_1 + \int^{p_0}\!f_0 + D_1 \left( \int^{p_1} \! L_1 + \int^{p_0} \! L_0 + \int^x \! L_{-1} \right)\ .
\end{cases}\end{equation*}
Putting these back into the above equalities, we obtain the required formula for $f_{2n-1}$ and
\begin{equation*}
\begin{split}
L_n=\,
    & (-1)^n \int^{p_n} \!\!\! \int \! e^{-R_n} + \sum^{n-1}_{m=2} \left( (-1)^m \int^{p_n} \!\!\! \int \! f_m + D_n \int^{p_m} \! L_m \right) \\
    & - \int^{p_1} \!\!\! \int \! f_1 + \int^{p_0}\!f_0 + D_2 \left( \int^{p_1} \! L_1 + \int^{p_0} \! L_0 + \int^x \! L_{-1} \right) \\
=\, & (-1)^n \int^{p_n} \!\!\! \int \! e^{-R_n} + \sum^{n-1}_{m=1} (-1)^m \int^{p_n} \!\!\!\int f_m + \int^{p_0}\!f_0 + D_n N_{n-1},
\end{split}
\end{equation*}
where we have let $\displaystyle N_{n-1} = \sum^{n-1}_{m=1} \int^{p_m} \! L_m + \int^{p_0} \! L_0 + \int^x \! L_{-1}$.

(d) $\Rightarrow$ (c): For each $m=n-1,n-2,\cdots,2$, take $L_m=\p_mN_m$ and let $N_{m-1}$ be such that
\begin{equation*}
N_m=\int^{p_m}\!L_m+N_{m-1}.
\end{equation*}
Then take also $L_1=\p_1N_1$ and
\begin{equation*}
\widetilde{L}_1=\p_0N_1\cdot p_1+\p_xN_1-\int^{p_1}\!\!\!\int\!f_1+\int^{p_0}\!f_0.
\end{equation*}
Using the calculations in the implication (c) $\Rightarrow$ (d), one can check that
\begin{equation*}
E_2^2\left(L_1p_2+\widetilde{L}_1\right)=f_1p_2-E_1^1\int^{p_1}\!\!\!\int\!f_1+f_0
\end{equation*}
and
\begin{equation*}
D_nN_{n-1}-\int^{p_1}\!\!\!\int\!f_1+\int^{p_0}\!f_0=D_n\sum_{n=2}^{n-1}\int^{p_m}\!L_m+L_1p_2+\widetilde{L}_1.
\end{equation*}
These equalities imply \eqref{thmBpf:eq2}. The proof is completed.
\end{proof}

%
%

%
%
\section{Relaxation on the Order of the Lagrangian}

The formulation of the multiplier problem in \eqref{1.4} is partially based on the assumption that we are looking for a non-degenerate Lagrangian and a multiplier for \eqref{1.2}.  Taking into account possible degeneracy in the Lagrangian, here we consider a slightly general situation, namely, the Lagrangian could have order greater than $n$ when the order of the equation is $2n$.  It turns out that the form of the solution of the multiplier problem remains almost the same. In other words, we do not recover extra variational differential equations of the form \eqref{1.2} by considering degenerate Lagrangians. Theorems \ref{thmC} and \ref{thmD} are generalizations of Theorems \ref{thmA} and \ref{thmB} respectively.

%
%
\begin{thm}\label{thmC}
For $m\geq 2$, suppose \begin{equation}\label{thmC:eq}
\rho_3(p_4-f_3)=E_{2m}^m L_m
\end{equation}
has a solution $L_m$ and $\rho_3$.  Then there are functions $R_2, f_0, f_1$ and $N_{m-1}$
such that \begin{equation}\label{thmC:sol}
\left\{
\begin{array}{rcl}
f_3 & = & \displaystyle \p_2R_2\cdot p^2_3 + 2D_2R_2\cdot p_3 \ ,\\
& & \displaystyle - e^{R_2} \left(E^2_2 \int^{p_2} \!\!\! \int \! e^{-R_2} + f_1p_2 -E^1_1 \int^{p_1} \!\!\! \int \! f_1+f_0 \right)\ , \\
\rho_2 & = & e^{-R_2} \\
L_2 & = & \displaystyle \int^{p_2} \!\!\! \int \! e^{-R_2} - \int^{p_1} \!\!\! \int \! f_1 + \int^{p_0} \! f_0 + D_m N_{m-1} \ .
\end{array}
\right.
\end{equation} Conversely, given any functions $R_2, f_0, f_1$ and $N_{m-1}$, the functions defined in \eqref{thmC:sol} satisfy \eqref{thmC:eq}.
\end{thm}
%
%

%
%
\begin{thm}\label{thmD}
For $m\geq n\geq 3$, suppose that
\begin{equation}\label{thmD:eq}
\rho_{2n-1}(p_{2n}-f_{2n-1}) = E_{2m}^m L_m
\end{equation}
has a solution $L_m$ and $\rho_{2n-1}$.  Then there are functions $R_n$, $(f_k)_{0\leq{k}\leq{n-1}}$ and $N_{m-1}$ such that \begin{equation}\label{thmD:sol}
\left\{
\begin{array}{rcl}
f_{2n-1} & = & \displaystyle nD_{n+1}R_n \cdot p_{2n-1} - e^{R_n} \left[ (-1)^n E^n_{2n-2} \int^{p_n} \!\!\! \int \! e^{-R_n}\right. \\
& & \displaystyle \left.+ \sum^{n-1}_{\ell=1} \left(f_{\ell} p_{2\ell} + (-1)^{\ell} E^{\ell}_{2\ell -1} \int^{p_n} \!\!\! \int \! f_{\ell} \right) + f_0 \right]\ , \\
\rho_{2n-1} & = & e^{-R_n} \ ,\\
L_m & = & \displaystyle (-1)^n \int^{p_n} \!\!\! \int \! e^{-R_n} + \sum_{\ell=1}^{n-1} (-1)^{\ell} \int^{p_\ell} \!\!\! \int \! f_{\ell} + \int^{p_0} \! f_0 + D_m N_{m-1} \ .
\end{array}
\right.
\end{equation}
Conversely, given any functions $R_n$, $(f_k)_{0\leq{k}\leq{n-1}}$ and $N_{m-1}$, the functions defined in \eqref{thmD:sol} satisfy \eqref{thmD:eq}.
\end{thm}
%
%

The following consequence of Lemma \ref{2.5} is essential in the proof of these theorems.
%
%
\begin{lem}\label{lem:E_{2n-2}^n,E_{2n-2}^{n-1}}
For $n\geq 2$,
\begin{equation*}
E^n_{2n-2} \left( L_{n-1} p_n + \widetilde{L}_{n-1} \right) = E^{n-1}_{2n-2} \left( \widetilde{L}_{n-1} - D_{n-1} \int^{p_{n-1}} \! L_{n-1} \right).
\end{equation*}
\end{lem}
%
%

%
%
\begin{proof}
By Lemma \ref{2.5},
\begin{equation*}
\begin{split}
  E^n_{2n-2} \left(L_{n-1} p_n + \widetilde{L}_{n-1} \right)
&=E^{n-1}_{2n-2} \left(L_{n-1}p_n + \widetilde{L}_{n-1} \right) + (-1)^n D^n_{2n-2} L_{n-1} \\
&=E^{n-1}_{2n-2} \left(L_{n-1}p_n + \widetilde{L}_{n-1} \right) - E^{n-1}_{2n-2} D_{2n-2} \int^{p_{n-1}} \! L_{n-1} \\
&=E^{n-1}_{2n-2} \left(\widetilde{L}_{n-1} - D_{n-1} \int^{p_{n-1}}\! L_{n-1} \right).
\end{split}
\end{equation*}
\end{proof}
%
%

Besides, the calculations in the proof of (a) $\Rightarrow$ (b) in Theorem \ref{thmB} is useful. We formulate it  as a lemma.

%
%
\begin{lem}\label{lem:E_{2n}^n}
For $n\geq 3$, we have
\begin{equation*}
\begin{split}
  E_{2n}^n L_n
&=(-1)^np_{2n}\p_n^2L_n+E_{2n-1}^nL_n\\
&=(-1)^np_{2n}\p_n^2L_n+(-1)^nnp_{2n-1}D_{2n-2}\p_n^2L_n+E_{2n-2}^nL_n.
\end{split}
\end{equation*}
\end{lem}
%
%

Since the proofs of Theorems \ref{thmC} and \ref{thmD} are essentially the same, we will only prove Theorem \ref{thmD}.

%
%
\begin{proof}
We proceed by induction on $m$. When $m=n$,  \eqref{thmD:eq} reads as
\begin{equation*}
\rho_{2n-1} ( p_{2n} - f_{2n-1}) = E_{2n-1}^n L_n + (-1)^n p_{2n} \p_n^2 L_n,
\end{equation*}
according to Lemma \ref{lem:E_{2n}^n}.  Hence $\rho_{2n-1} = (-1)^n \p_n^2 L_n$ depends only up to $p_n$ and we can write $\rho_{2n-1}=\rho_n$. Now, Theorem \ref{thmB} gives the solution of \eqref{thmD:eq} in the stated form.

Assuming the result holds for some $m\geq n$, we consider the equation
\begin{equation*}
\rho_{2n-1} (p_{2n} - f_{2n-1}) = E^{m+1}_{2m+2} L_{m+1}.
\end{equation*}
The left hand side depends on only up to $p_{2n}$ while the right hand side can be expanded, using Lemma \ref{lem:E_{2n}^n} again, as
\begin{align*}
E^{m+1}_{2m+2} L_{m+1} =\,
& E^{m+1}_{2m} L_{m+1} + (-1)^{m+1} p_{2m+2} \p^2_{m+1} L_{m+1} \\
& + (-1)^{m+1} (m+1) p_{2m+1} D_{2m} \p^2_{m+1} L_{m+1}.
\end{align*}
Since $2m+2 \geq 2n+2 > 2n$, comparing the coefficients of $p_{2m+2}$ yields
$$\p^2_{m+1} L_{m+1} = 0.$$
Thus $E^{m+1}_{2m+2} L_{m+1} = E^{m+1}_{2m} L_{m+1}$ and there exist $L_m$ and $\widetilde{L}_m$ such that
\begin{equation*}
L_{m+1} = L_m p_{m+1} + \widetilde{L}_m.
\end{equation*}
By Lemma \ref{lem:E_{2n-2}^n,E_{2n-2}^{n-1}}, we arrive at
\begin{equation*}
\displaystyle \rho_{2n-1} (p_{2n} -f_{2n-1}) = E^m_{2m} \left( \widetilde{L}_m - D_m \int^{p_m} \! L_m \right).
\end{equation*}
By the induction hypothesis, it admits a solution where $f_{2n-1},L_m$ and $\rho_{2n-1}$ are given by \begin{equation*}
\left\{
\begin{array}{rl}
f_{2n-1}
=& \!\!\!\displaystyle nD_{n+1}R_n \cdot p_{2n-1} - e^{R_n} \left[ (-1)^n E^n_{2n-2} \int^{p_n} \!\!\! \int \! e^{-R_n} \right. \\
 & \displaystyle \left. + \sum^{n-1}_{j =1} \left(f_{j} p_{2\ell} + (-1)^{j} E^{\ell}_{2j -1} \int^{p_n} \!\!\! \int \! f_{j} \right) + f_0 \right], \\
\rho_{2n-1}
=& \!\!\! e^{-R_n}, \\
\displaystyle \widetilde{L}_m - D_m \int^p \! L_m
=& \!\!\! \displaystyle (-1)^n \int^{p_n} \!\!\! \int \! e^{-R_n} + \sum_{j=1}^{n-1} (-1)^{j} \int^{p_j} \!\!\! \int \! f_{j} + \int^{p_0} \! f_0 + D_m N_{m-1},
\end{array}
\right.
\end{equation*}
where $R_n$, $(f_j)_{0\leq{j}\leq{n-1}}$ and $N_{m-1}$ are free. The inductive step is completed upon observing \begin{align*}
& L_{m+1}\\
=\, & \left( L_m p_{m+1} + D_m \int^{p_m} \! L_m \right) + \left( \widetilde{L}_m - D_m \int^{p_m} \! L_m \right) \\
=\, & \displaystyle D_{m+1} \int^{p_m} \! L_m + (-1)^n \int^{p_n} \!\!\! \int \! e^{-R_n} + \sum_{j=1}^{n-1} (-1)^{j} \int^{p_j} \!\!\! \int \! f_{j} + \int^{p_0} \! f_0 + D_m N_{m-1} \\
=\, & (-1)^n \int^{p_n} \!\!\! \int \! e^{-R_n} + \sum^{n-1}_{j =1} (-1)^{j} \int^{p_j} \!\!\! \int \! f_{j} + \int^{p_0} \! f_0 + D_{m+1} \left( \int^{p_m} \! L_m + N_{m-1} \right)
\end{align*}
and letting $\displaystyle N_m = \int^{p_m} \! L_m + N_{m-1}$.
\end{proof}
%

%
%

\section{An Algorithm for Variationality}
Based on the proof of Lemma \ref{2.6} and the form of the solution stated in Theorems \ref{thmA} and \ref{thmB}, here we describe an algorithm to determine whether a given differential equation \eqref{1.2}, or, $u^{(2n)}=f_{2n-1}[u],$ admits a variational multiplier. It is affirmative
only if all the following steps are checked successively.

In this algorithm, Step (1) checks whether the $(2n-1)$-th derivative of $f_{2n-1}$ depends only up to $p_{n+1}$. Steps (2) and (3) check the existence of $R_n$. Note that $\p_kg_k$ and $g_k-\p_kg_k\cdot p_k$ correspond respectively to $\alpha_{k-1}$ and $\beta_{k-1}$ in the proof of Lemma \ref{2.6}. In the rest of the algorithm, we may think of $h_{2n-2-2j}$ ($0\leq{j}\leq{n-1}$) as
\begin{equation*}
\sum_{m=1}^j\left(f_mp_{2m}+(-1)^mE_{2m-1}^m\int^{p_m}\!\!\!\int\!f_m\right)+f_0,
\end{equation*}
thus $\p_{2n-2-2j} h_{2n-2-2j}$ as $f_{n-1-j}$.

\medskip

{\bf Algorithm}
\begin{description}
\item[Step (1)] If $n\geq3$, check whether $\p_k\p_{2n-1}f_{2n-1}=0$, for all $n+2\leq{k}\leq{2n-1}$. If so, write $$g_{n+1}=\dfrac{1}{n}\p_{2n-1}f_{2n-1}.$$
    If $n=2$, define
    $$g_3=\dfrac12\p_3f_3$$
    and proceed directly to Step (2).
\item[Step (2)] For $k=n+1,n,\dots, 1$, perform the followings:
    \begin{description}
    \item[Substep ($n+2-k$)] Check whether $\p_k^2 g_k=0$. If so, write
        $$g_{k-1}=g_k-\p_kg_k\cdot p_k-D_{k-1}\int^{p_{k-1}}\!\!\!\p_kg_k.$$
    \end{description}
\item[Step (3)] Check whether $\p_0g_0=0$. If so, write
    $$R_n=\sum_{k=1}^n\int^{p_k}\!\p_kg_k+\int^x\!\left(g_1-\p_1g_1\cdot p_1\right)$$
    and
    \begin{equation*}
    \begin{split}
    h_{2n-2}=\,&-e^{R_n}\left(f_{2n-1}-\p_{2n-1}f_{2n-1}\cdot p_{2n-1}+\dfrac{1}{2}\p_{2n-1}^2f_{2n-1}\cdot p_{2n-1}^2\right)\\
               &-(-1)^nE_{2n-2}^n\int^{p_n}\!\!\!\int\!e^{-R_n}.
    \end{split}
    \end{equation*}
\item[Step (4)] For $j=1,2,\dots,n-1$, perform the followings:
    \begin{description}
    \item[Substep ($j$)] Check whether $\p_{2n+1-2j}h_{2n-2j}=0$ and $\p_k\p_{2n-2j}h_{2n-2j}=0$, for all $n+1-j\leq{k}\leq{2n-2j}$. If so, write
        \begin{equation*}
        \begin{split}
        h_{2n-2-2j}=\,&h_{2n-2j}-\p_{2n-2j}h_{2n-2j}\cdot p_{2n-2j}\\
                         &-(-1)^{n-j}E_{2n-1-2j}^{n-j}\int^{p_{n-1}}\!\!\!\!\int\!\p_{2n-2j}h_{2n-2j}.
        \end{split}
        \end{equation*}
    \end{description}
\item[Step (5)] Check whether $\p_1h_0=0$. If so, all is done and the given differential equation admits a variational multiplier $\rho_n=e^{-R_n}$.
\end{description}

\begin{remark}
We remark that, as opposed to Fels and Jur\'{a}\v{s}'s approach, we have \emph{more} (in fact, $n^2+n+1$), yet \emph{linear} in the highest order, vanishing quantities in the necessary and sufficient condition for \eqref{1.2} to be variational. 
\end{remark}

\section{An Application in Parabolic PDE's}

We were led to the multiplier problem by the study of the parabolic problem
\begin{equation}\label{8.1}
u_t=a(x,u,u_x)u_{xx}+b(x,u,u_x), \  \ x\in [0,1],\ t>0,
\end{equation}
where $a$ is always positive.  The equation is subject to the boundary conditions
\begin{equation}\label{8.2}
\alpha_1u(0,t) +\beta_1 u_x(0,t)=0, \ \alpha_2u(1,t) +\beta_2 u_x(1,t)=0,\ \alpha_i^2+\beta_i^2>0, \ i=0,1.
\end{equation}
A well-known result of Zelenyak asserts the following.

\begin{thm*}[Zelenyak\cite{Z}]
Any globally bounded solution of \eqref{8.1}--\eqref{8.2} converges uniformly to a steady state.
\end{thm*}

An essential ingredient in his proof is that there is always some $L(x,u,u_x)$ and
positive
$\rho(x,u,u_x)$ such that the first variation for $$F[u]=\int L(x,u,u_x)dx$$ along $\varphi$ satisfies
$$F'[u]\varphi = - \int \rho\big(au_{xx}+b\big)\varphi dx,$$ for all $u$ satisfying \eqref{8.2}. Consequently, $L$ is a Liapunov function for \eqref{8.1}:
$$\dfrac{d }{dt}\int L(x,u,u_x) dx = -\int \rho u_t^2\leq 0.$$
Now, consider the fourth order quasi-linear equation
$$u_t= -a(x,u,u_x,u_{xx},u_{xxx})u_{xxxx} + b(x,u,u_x,u_{xx},u_{xxx}),\ \ x\in[0,1],\ t>0,$$ together with some boundary conditions.
When there is a positive multiplier for this problem, one may construct a Liapunov functional and extend Zelenyak's argument to
fourth order equations.  In the inverse problem of the calculus of variations, the boundary integrals arising in the derivation of the first variation formula are usually ignored.  When applying to initial-boundary value problems, it works only if the boundary conditions are periodic.  How to construct a multiplier whose corresponding Lagrangian adapts to some given boundary conditions is an interesting but rarely explored topic.  Nevertheless, we point out that the Lagrangian and multiplier constructed in Theorem \ref{thmA} is adapted to the Dirichlet boundary conditions $u=u_x=0$ at $x=0,1$.  Using this fact, we are able to obtain a Zelenyak-type theorem for fourth order equations, see \cite{C} for details. It is also clear that similar theorems can be obtained in the higher order cases. The details are left to the readers.

\newpage

\appendix
\section{Proofs of Some Elementary Results}
\subsection{Proof of Lemma \ref{2.3}}

\begin{proof}
Let $m, n, m\geq n,$ be fixed. Note that this lemma is trivial for $k=0$. If it holds for some $k<n$, then
\begin{align*}
\p_n D_m^{k+1} =\, & \sum_{j=0}^k \binom{k}{j} D_m^{k-j} \p_{n-j} D_m \\
               =\, & \sum_{j=0}^k \binom{k}{j} D_m^{k-j} (D_m \p_{n-j} + \p_{n-j-1}) \\
               =\, & \sum_{j=0}^k \binom{k}{j} D_m^{k+1-j} \p_{n-j} + \sum_{j=1}^{k+1} \binom{k}{j-1} D_m^{k+1-j} \p_{n-j} \\
               =\, & D^{k+1}_m \p_n + \sum^k_{j=1} \left[ \binom{k}{j} + \binom{k}{j-1} \right] D_m^{k+1-j} \p_{n-j} + \p_{n-k-1} \\
               =\, & \sum^{k+1}_{j=0} \binom{k+1}{j} D_m^{k+1-j} \p_{n-j}\ .
\end{align*}
Hence the formula holds for all $k=0,1,\dots,n$. On the other hand, if it is true for some $k\geq n$, we have
\begin{align*}
\p_n D_m^{k+1} =\, & \sum_{j=0}^n \binom{k}{j} D_m^{k-j} \p_{n-j} D_m \\
               =\, & \sum_{j=0}^{n-1} \binom{k}{j} D_m^{k-j} (D_m \p_{n-j} + \p_{n-j-1}) + \binom{k}{n} D^{k-n}_m \p_0 D_m \\
               =\, & \sum_{j=0}^{n-1} \binom{k}{j} D_m^{k+1-j} \p_{n-j} + \sum_{j=0}^{n-1} \binom{k}{j} D_m^{k-j} \p_{n-j-1} + \binom{k}{n} D^{k+1-n}_m \p_0 \\
               =\, & D^{k+1}_m \p_n + \sum_{j=1}^n \binom{k}{j} D_m^{k+1-j} \p_{n-j} + \sum_{j=1}^n \binom{k}{j-1} D_m^{k+1-j} \p_{n-j} \\
               =\, & D^{k+1}_m \p_n + \sum_{j=1}^n \binom{k+1}{j} D_m^{k+1-j} \p_{n-j} \\
               =\, & \sum_{j=0}^n \binom{k+1}{j} D_m^{k+1-j} \p_{n-j},
\end{align*}
and the result follows.
\end{proof}

\medskip

\subsection{Analytic Proof of Lemma \ref{2.4}}

\begin{proof}
Let $m$ be fixed and we proceed by induction on $k$. For $k=1$, if
\begin{equation*}
\norm{I} = i_{m-1} + 2i_{m-2} + \cdots + mi_0 \leq 1,
\end{equation*}
the only possible multi-indices are $I^m = (0,0,\dots,0)$ and $I^m = (1,0,\ldots,0)$. Then the right hand side reduces to $D_m$, agreeing with the left hand side.

Assume that the result holds for some $1\leq{k}\leq{m-2}$. Then,
\begin{align*}
D_m^{k+1} =\, & (D_{m-1} + p_m \p_{m-1} ) \sum_{\norm{I} \leq k} a_{I}^{(k)} p^{\abs{I}}_m D^{k-\norm{I}}_{m-1} \p^{I} \\
          =\, & \sum_{\norm{I} \leq k} a_{I}^{(k)} p^{\abs{I}}_m D^{k+1-\norm{I}}_{m-1} \p^{I} + \sum_{\norm{I} \leq k} a_{I}^{(k)} p^{\abs{I}+1}_m \p_{m-1} D^{k-\norm{I}}_{m-1} \p^{I}.
\end{align*}
The idea is simple --- use Lemma \ref{2.3} to move the single derivative $\p_{m-1}$ in the second term to the far right and perform a shift in the ``dummy'' multi-index $I$. This is possible since the coefficient $a_I^{(k)}\neq0$ only when $I\geq0$ and $\binom{n}{r}\neq0$ only in the case $n\geq{r}$, for integers $n$ and $r$.

To carry out this argument, we observe that by Lemma \ref{2.3},
$$\p_{m-1}D_{m-1}^{k-\norm{I}}=\sum_{j=0}^{k-\norm{I}} \binom{k-\norm{I}}{j}D_{m-1}^{k-\norm{I}-j}\p_{m-1-j}$$
since $k-\norm{I}\leq k\leq m-1$. However, since $\displaystyle\binom{k-\norm{I}}{j}=0$ for $j>k-\norm{I}$, we also have
\begin{align*}
\p_{m-1}D_{m-1}^{k-\norm{I}} =\, & \sum_{j=0}^{m-1} \binom{k-\norm{I}}{j}D_{m-1}^{k-\norm{I}-j}\p_{m-1-j} \\
                               =\, & \sum_{j=1}^{m} \binom{k-\norm{I}}{j-1}D_{m-1}^{k+1-\norm{I}-j}\p_{m-j}.
\end{align*}
It follows that
\begin{align*}
D_m^{k+1} =\, & \sum_{\norm{I}\leq k} a_{I}^{(k)}p^{\abs{I}}_m D^{k+1-\norm{I}}_{m-1} \p^{I} 
               + \sum_{j=1}^m\sum_{\norm{I}\leq k} \binom{k-\norm{I}}{j-1}a_{I}^{(k)} p_m^{\abs{I}+1}D_{m-1}^{k+1-\norm{I}-j}\p_{m-j}\p^{I}. \\
\end{align*}
By writing $\p_{m-j}=\p^{e_j}$ where $e_j=(0,\dots,1,\dots,0)$, with its ``1'' appears in the $j$-th slot from the left, we have
\begin{align*}
D_{m}^{k+1}=\, & \sum_{\norm{I}\leq k}
               a_{I}^{(k)}p^{\abs{I}}_m D^{k+1-\norm{I}}_{m-1} \p^{I} \\
              & + \sum_{j=1}^m\sum_{\norm{I+e_{j}}\leq k+j} \binom{k+j-\norm{I+e_{j}}}{j-1}a^{(k)}_{(I+e_{j})-e_{j}} \cdot p_m^{\abs{I+e_{j}}}D_{m-1}^{k+1-\norm{I+e_{j}}}\p^{I+e_{j}}. \\
\end{align*}
Since we have set $a_{I}^{(k)}=0$ whenever $i_{m-j}<0$ for some $j$ with $1\leq{j}\leq{m}$, we may implicitly assume that $I-e_j\geq0$ in the summands. Hence,
\begin{align*}
D_m^{k+1} =\, & \sum_{\norm{I}\leq k} a_{I}^{(k)} p_m^{\abs{I}}D_{m-1}^{k+1-\norm{I}}\p^{I} 
               + \sum_{j=1}^{m} \sum_{\norm{I}\leq k+j} \binom{k+j-\norm{I}}{j-1} a_{I-e_{j}}^{(k)}p_m^{\abs{I}}D_{m-1}^{k+1-\norm{I}}\p^{I}.
\end{align*}
Using the fact that $\displaystyle \binom{k+j-\norm{I}}{j-1}=0$ for $\norm{I}>k+1$, we have
\begin{align*}
D_m^{k+1} =\, & \sum_{\norm{I}\leq k} a_{I}^{(k)} p_m^{\abs{I}}D_{m-1}^{k+1-\norm{I}}\p^{I} 
               + \sum_{j=1}^{m} \sum_{\norm{I}\leq k+1} \binom{k+j-\norm{I}}{j-1} a_{I-e_{j}}^{(k)}p_m^{\abs{I}}D_{m-1}^{k+1-\norm{I}}\p^{I} \\
          =\, & \sum_{\norm{I}\leq k} \left[a_{I}^{(k)}+\sum_{j=1}^{m}\binom{k+j-\norm{I}}{j-1}a^{(k)}_{I-e_{j}}\right]p_m^{\abs{I}}D_{m-1}^{k+1-\norm{I}}\p^{I}.
\end{align*}
Finally, we compute
\begin{align*}
    & a_{I}^{(k)}+\sum_{j=1}^{m} \binom{k+j-\norm{I}}{j-1}a_{I-e_{j}}^{(k)} \\
=\, & \dfrac{k!}{\norm{I}^*I!(k-\norm{I})!} 
    +\sum_{j=1}^{m} \dfrac{(k+j-\norm{I})!}{(j-1)!(k+1-\norm{I})!}\cdot\dfrac{k!}{\frac{\norm{I}^*}{j!}\cdot\frac{I!}{i_{m-j}}\cdot(k+j-\norm{I})!} \\
=\, & \dfrac{k!}{\norm{I}^*I!(k-\norm{I})!} + \sum_{j=1}^{m} \dfrac{ji_{m-j}k!}{\norm{I}^*I^m!(k+1-\norm{I})!} \\
=\, & \dfrac{k!}{\norm{I}^*I!(k+1-\norm{I})!} \left(k+1-\norm{I}+\sum_{j=1}^{m} ji_{m-j} \right) \\
=\, & \dfrac{(k+1)!}{\norm{I}^*I!(k+1-\norm{I})!} \\
=\, & a^{(k+1)}_{I}.
\end{align*}
Since the result is also true for $k+1$, it holds for all $k=1,\dots,m-1$.
\end{proof}

\subsection{Combinatorial Proof of Lemma \ref{2.4}}
\begin{proof}
By Lemma \ref{2.2}(a), $D_m^k$ can be expanded as a binomial
\begin{equation*}\begin{split}
D_m^k
&=(D_{m-1}+p_m\p_{m-1})^k\\
\end{split}\end{equation*}
which is equal to the sum of all possible strings of length $k$ formed by the alphabet $\{D_{m-1},p_m\p_{m-1}\}$. Each $\p_{m-1}$ followed by $D_{m-1}$ gives rise to more terms by Lemma \ref{2.2}(b). In what follows, in the commutation of $p_m\p_{m-1}D_{m-1}$, we call $p_mD_{m-1}\p_{m-1}$ the higher order term and $p_m\p_{m-2}$ the lower order one. Hence,
$$p_m\p_{m-1}D_{m-1}^{j-1}=p_m\p_{m-j}+\text{higher order terms}.$$

Suppose a term has partial derivatives $\p^I=\p_{m-j}^{i_{m-j}}\cdots\p_0^{i_0}$. Two observations are in order:
\begin{itemize}
\item Each $\p_{m-j}$ is associated to one and only one $p_m$. Hence $\p^I$ is multiplied by $p_m^{\abs{I}}$.
\item Each $\p_{m-j}$ is obtained as the lowest order term of $\p_{m-1}D_{m-1}^{j-1}$. The resulting number of $D_{m-1}$ plus $\norm{I}$ is then the constant $k$.
\end{itemize}
It suffices to count the number of times $p_m^{\abs{I}}D_{m-1}^{k-\norm{I}}\p^I$ appears.

We observe that the coefficient
$$\dfrac{k!}{\norm{I}^*I!}=a_I^{(k)}(k-\norm{I})!$$
counts the number of ways partitioning a $k$-element set into sub-classes whose numbers of elements are given by $I$, see, for example, 12(a) in Chapter 1 of Lov\'{a}sz\cite{L}. By dividing it by $(k-\norm{I})!$, the number that $D_{m-1}$ appears, we do not distinguish between the $D_{m-1}$'s. Therefore, $a_I^{(k)}$ is the number of ways forming permutations of $i_{m-1}$ many $(p_m\p_{m-1})$'s, $i_{m-2}$ many $(p_m\p_{m-1}D_{m-1})$'s, $i_{m-3}$ many $(p_m\p_{m-1}D_{m-1}^2)$'s, etc., as well as $(k-i_{m-1}-2i_{m-2}-\cdots-mi_0)$ many $D_{m-1}$'s.

In this setting, $I=(i_{m-1},\dots,i_{0})$ where $i_{m-j}$ is the number of $p_{m}\p_{m-1}$'s to be commuted with its (immediately) following $(j-1)$ $D_{m-1}$'s. To fix the idea, let us look at an example. Suppose $k=10$ and $I=(i_{m-1},i_{m-2},i_{m-3})=(3,1,1)$, i.e. consider
\begin{equation}\label{eq2.1}
p_m^5D_{m-1}^2\p_{m-1}^3\p_{m-2}\p_{m-3}.
\end{equation}
This could be a result of re-ordering (always taking the highest order terms), say,
\begin{equation}\label{eq2.2}
D_{m-1}(p_m\p_{m-2})D_{m-1}(p_m\p_{m-1})^3(p_m\p_{m-3}).
\end{equation}
In order that it comes into place, the lower order $\p_{m-2}$ and $\p_{m-3}$ must be originally $\p_{m-1}D_{m-1}$ and $\p_{m-1}D_{m-1}^2$. Thus the corresponding binomial expansion is
\begin{equation}\label{eq2.3}
D_{m-1}(p_m\p_{m-1}D_{m-1})D_{m-1}(p_m\p_{m-1})^3(p_m\p_{m-1}D_{m-1}^2),
\end{equation}
or
$$D_{m-1}(p_{m}\p_{m-1})D_{m-1}^2(p_{m}\p_{m-1})^4D_{m-1}^2,$$
one of the many terms in $(D_{m-1}+p_m\p_{m-1})^{10}$.

Since \eqref{eq2.2} has a one-to-one correspondence to \eqref{eq2.3} when $I$ is fixed, the total number of terms of the form \eqref{eq2.1} is given by $a_I^{(k)}$, as desired.
\end{proof}



\begin{thebibliography}{10}

\bibitem{AD1}
Anderson I. and Duchamp, T.E., On the existence of global variational principles, Amer. J. Math. 102(1980), 781-868.

\bibitem{AD2}
Anderson I. and Duchamp, T.E., Variational principles for second-order quasi-linear scalar equations, J. Diff. Eq. 51(1984), 1-47.

\bibitem{AT}
Anderson I. and Thompson G., The inverse problem of the calculus of variations for ordinary differential equations, Mem. Amer. Math. Soc. 98, no. 473, 1992.

\bibitem{C}
Chan, H.-T. H., Convergence of bounded solutions for nonlinear parabolic equations, M. Phil. Thesis, Chinese Univ. of Hong Kong, 2012.

\bibitem{D}
Darboux, G., Le\c{c}on sur la th\'{e}orie g\'{e}n\'{e}rale des surfaces, Vol III, Gauthier-Villars, Paris, 1894.

\bibitem{DZ}
Doubrov, B. and Zelenko, I.,  Equivalence of variational problems of higher order, Differential Geom. Appl. 29(2011), no. 2, 255-270.


\bibitem{F}
Fels, M.E., The inverse problem of the calculus of variations for scalar fourth-order ordinary differential equations, Trans. Amer. Math. Soc. 348 (1996), 5007-5029.


\bibitem{J1}
Jur\'{a}\v{s}, M., The inverse problem of the calculus fo variations for sixth- and eighth-order scalar ordinary differential equations, Acta Appl. Math. 66(2001), 25-39.

\bibitem{J2}
Jur\'{a}\v{s}, M., Towards a solution of the inverse problem of the calculus of variations for scalar ordinary differential equations, Differential geometry and its applications (Opava, 2001), 425-434,
Math. Publ., 3, Silesian Univ. Opava, Opava, 2001.

\bibitem{L}
Lov\'{a}sz, L., Combinatorial problems and exercises. Second edition. North-Holland Publishing Co., Amsterdam, 1993. 635 pp. ISBN: 0-444-81504-X.
\bibitem{NA}
Nucci, M. C. and Arthurs, A. M.,  On the inverse problem of calculus of variations for fourth-order equations,
Proc. R. Soc. Lond. Ser. A Math. Phys. Eng. Sci. 466(2010), 2309-2323.

\bibitem{O}
Olver, P. J., Applications of Lie Groups to Differential Equations, 2nd ed., Springer-Verlag, New York, 1993.

\bibitem{S}
Saunders, D. J.,  Thirty years of the inverse problem in the calculus of variations, Rep. Math. Phys. 66(2010), no. 1, 43-53.


\bibitem{Z}
Zelenyak, T. I., Stabilization of solutions of boundary value problems for a second-order parabolic equation with one space variable,
Differential Equations 4(1968), 17-22.

\end{thebibliography}
\end{document}